\documentclass[fleqn,final]{zzz}
\usepackage{mathrsfs}
\usepackage{color}
\usepackage{tikz}
\usepackage{float}
\usepackage{graphicx}
\usepackage{rotating}
\usepackage[pdftex]{hyperref}
\usepackage{enumerate}
\usepackage{soul}
\usepackage{comment}
\usepackage{xparse}
\usepackage{chngcntr}
\usepackage{apptools}
\usepackage{caption}
\usepackage[normalem]{ulem}
\usepackage{subcaption}
\usepackage{arydshln}

\hypersetup{
colorlinks=true,
urlcolor=blue,
linkcolor=red,
citecolor=green
}

\newcommand\boro[1]{{\textcolor{black}{#1}}}
\newcommand\TT{\rule{0pt}{2.9ex}}
\newcommand\TB{\rule[-1.5ex]{0pt}{0pt}}

\newcommand{\Whyp}[5]{\,\mbox{}_{#1}W_{#2}\!\left({#3};{#4};{#5}\right)}
\newcommand{\qhyp}[5]{\,\mbox{}_{#1}\phi_{#2}\!\left(
\genfrac{}{}{0pt}{}{#3}{#4};#5\right)}

\NewDocumentCommand{\qfrac}{smm}{%
 \dfrac{\IfBooleanT{#1}{\vphantom{\big|}}#2}{\mathstrut #3}%
}

\newcommand{\nqphyp}[3]{\,\sideset{_{#1}^{\phantom{\mid}}}{_{#2}^{#3}}{\mathop{\phi}}}
\newcommand{\nlqphyp}[3]{\,\sideset{_{#1}}{_{#2}^{#3}}{\mathop{\phi}}}
\newcommand{\rphis}[2]{{}_{#1\vphantom{#2}}\phi_{#2\vphantom{#1}}}
\newcommand{\rphisx}[4]{\rphis{#1}{#2}\left( \begin{array}{c} #3 \end{array};q,#4\right)}

\newcommand{\qphyp}[6]{\,\sideset{_{#1}^{\phantom{\mid}}}{_{#2}^{#3}}{\mathop{\phi}}\!\left(\genfrac{}{}{0pt}{}{#4}{#5};#6\right)}

\def\cprime{$'$}

\makeatletter
\def\eqnarray{\stepcounter{equation}\let\@currentlabel=\theequation
\global\@eqnswtrue
\tabskip\@centering\let\\=\@eqncr
$$\halign to \displaywidth\bgroup\hfil\global\@eqcnt\z@
  $\displaystyle\tabskip\z@{##}$&\global\@eqcnt\@ne
  \hfil$\displaystyle{{}##{}}$\hfil
  &\global\@eqcnt\tw@ $\displaystyle{##}$\hfil
  \tabskip\@centering&\llap{##}\tabskip\z@\cr}

\def\endeqnarray{\@@eqncr\egroup
      \global\advance\c@equation\m@ne$$\global\@ignoretrue}

\def\@yeqncr{\@ifnextchar [{\@xeqncr}{\@xeqncr[5pt]}}
\makeatother

\newcommand{\N}{{\mathbb N}}

\newcommand{\R}{{\mathbb R}}
\newcommand{\Z}{{\mathbb Z}}
\newcommand{\CC}{{\mathbb C}}
\newcommand{\expe}{{\mathrm e}}
\newcommand{\CCast}{{{\mathbb C}^\ast}}
\newcommand{\CCdag}{{{\mathbb C}^\dag}}
\newcommand{\SSS}{{\mathcal S}}

\newtheorem{thm}[lemma]{Theorem}
\newtheorem{cor}[lemma]{Corollary}
\newtheorem{rem}[lemma]{Remark}

\newtheorem{prop}[lemma]{Proposition}
\newtheorem{defn}[lemma]{Definition}

\newcommand{\mlt}[1]{ \begin{tabular}[c]{@{}c@{}}#1\end{tabular}}

\definecolor{darkgreen}{rgb}{0.0, 0.21, 0.06}

\begin{document}

\renewcommand{\PaperNumber}{***}

\FirstPageHeading

\ShortArticleName{%
%Symmetry of terminating basic hypergeometric representations of the Askey--Wilson Polynomials
Symmetry of terminating series representations of the Askey--Wilson polynomials
}

\ArticleName{%
%Symmetry of terminating basic hypergeometric\\ representations of the Askey--Wilson polynomials
Symmetry of terminating basic hypergeometric series representations of the Askey--Wilson polynomials
}

% Names of the authors for the title of the paper
\Author{Howard S. Cohl\,$^\dag\!\!\ $ and Roberto 
S. Costas-Santos\,$^\S\!\!$}

\AuthorNameForHeading{H.~S.~Cohl, R.~S.~Costas-Santos}
\Address{$^\dag$ Applied and Computational 
Mathematics Division, National Institute of Standards 
and Tech\-no\-lo\-gy, Mission Viejo, CA 92694, USA
%Address of First Author, Country
\URLaddressD{
\href{http://www.nist.gov/itl/math/msg/howard-s-cohl.cfm}
{http://www.nist.gov/itl/math/msg/howard-s-cohl.cfm}
}
} % Address of First Author
\EmailD{howard.cohl@nist.gov} % E-mail address of First Author

\Address{$^\S$ Dpto. de F\'isica y Matem\'{a}ticas,
Universidad de Alcal\'{a},
c.p. 28871, Alcal\'{a} de Henares, Spain} 
% Address of First Author
\URLaddressD{
\href{http://www.rscosan.com}
{http://www.rscosan.com}
}
\EmailD{rscosa@gmail.com} % E-mail address of First Author

%\Address{$^{\S\S}$ Department of Mathematics,
%University of Rochester, Rochester, NY 14627, USA
%% Address of First Author
%}
%\EmailD{random9483@gmail.com} % E-mail address of First Author

\ArticleDates{Received \today~in final form ????; 
Published online ????}

\Abstract{In this paper, we explore the symmetric nature of the terminating basic hypergeometric series
representations of the Askey--Wilson polynomials and the corresponding terminating basic 
hypergeometric transformations that these polynomials satisfy. In particular we identify and 
classify the set of 4 and 7 equivalence classes of terminating balanced ${}_4\phi_3$ and terminating 
very-well poised ${}_8W_7$ basic hypergeometric series which are connected with the Askey--Wilson polynomials. 
We study the inversion properties of these equivalence classes and also identify the connection of both 
sets of equivalence classes with the symmetric group $S_6$, the symmetry group of the terminating 
balanced ${}_4\phi_3$. We then use terminating balanced ${}_4\phi_3$ and terminating very-well 
poised ${}_8W_7$ transformations to give a broader interpretation of Watson's $q$-analog of Whipple's 
theorem and its converse. We give a broad description of the symmetry structure of the terminating
basic hypergeometric series representations of the Askey--Wilson polynomials.
} 

%``Symmetry, Integrability and Geometry: Methods and Applications''.}

%Please type here List of Keywords for your article separated 
%by semicolon.
% Keywords:

\Classification{33D15, 33D45}
%{??????} % e.g. 35A30; 81Q05
%For 2010 Mathematics Subject Classification see
%http://www.ams.org/mathscinet/msc/msc2010.html

%\dedicatory{} 
\begin{flushright}
\begin{minipage}{70mm}
\it Dedicated to the life and mathematics\\ of Dick Askey, 1933-2019.
\end{minipage}
\end{flushright}

\section{Introduction}
\label{Introduction}

This paper is a study in $q$-calculus (typically taken with $|q|<1$).
The $q$-calculus (introduced by such luminaries as Leonhard Euler,
Eduard Heine and Garl Gustav Jacobi)
is a calculus of finite differences which becomes the standard infinitesimal calculus (introduced by
Isaac Newton and Gottfried Wilhelm Leibniz) in the limit
as $q\to 1$. One of the most important aspects
of $q$-calculus is the theory of basic hypergeometric
series which are the $q$-analogue
of generalized hypergeometric series. \boro{Observe that these 
%are deeply connected to the theory of \goro{basic hypergeometric orthogonal polynomials which 
obey a natural scheme which is often 
referred to as the scheme of basic hypergeometric 
orthogonal polynomials}. 
\boro{Hereafter we refer to this scheme,
which represents a hierarchy
of basic hypergeometric orthogonal polynomials
(see e.g., \cite[p.~414]{Koekoeketal}),
as the $q$-Askey scheme, in honor of Dick Askey who 
was instrumental in the understanding and classification of
hypergeometric orthogonal polynomials.}
The Askey--Wilson
polynomials are at the very top of the $q$-Askey 
scheme and all polynomials within the $q$-Askey scheme can be
written as either a 
specialization or limit of the 
Askey--Wilson polynomials.
The work contained in this paper is
associated with the Askey--Wilson polynomials $p_n(x;{\bf a}|q)$
\cite[\S 14.1]{Koekoeketal}.
The Askey--Wilson polynomials
are basic hypergeometric orthogonal
polynomials with interpretations
in quantum group theory, 
combinatorics, and probability.
The applications of Askey--Wilson
polynomials include invariants of
links, ${\textit{3}}$-manifolds 
and ${\textit{6}}j$-symbols
(see e.g., \cite{Rosengren2007}).
The definition of the Askey--Wilson
polynomials in terms of terminating
basic hypergeometric series
is given in
Theorem \ref{AWthm} below.
%The polynomials in the $q$-Askey scheme are terminating basic 
%hypergeometric orthogonal polynomials 
%%which are defined in terms of terminating basic hypergeometric functions 
%(see e.g., \cite{GaspRah}).
%At the very top of the $q$-Askey scheme, lies the Askey--Wilson 
%polynomials 
%whose full definition can be found below 
The Askey--Wilson polynomials are symmetric with respect to their 
four free parameters, that is, they remain unchanged upon 
interchange of any two of the
four free parameters. 
\boro{It should be emphasized that since
1970, the subjects of special functions
and special families of
orthogonal polynomials have gone
through major developments, 
of which the study of the Askey--Wilson
polynomials has been central.
Many of the properties of these
polynomials can be derived from
their terminating basic hypergeometric
representations, so an exhaustive
catalog of these representations, as contained here,
will be quite convenient for lookup.}

From the terminating basic hypergeometric representations of the 
Askey--Wilson polynomials, one can easily 
derive transformation formulas for terminating basic hypergeometric functions.
%, using known properties of basic hypergeometric functions, basic 
%hypergeometric representations of the symmetric subfamilies of the 
%Askey--Wilson and the $q$-inverse Askey--Wilson polynomials.
The main focus of this survey paper will be to exhaustively describe the 
transformation identities for the terminating basic hypergeometric functions which
appear as representations for these polynomials. Some of these transformation 
identities are well-known in the literature, but we also give the transformation 
identities for these basic hypergeometric functions which are 
obtained by the symmetry of the polynomials under parameter interchange, 
and under the map $\theta\mapsto-\theta$, for $x=\cos\theta$.

This paper follows from the preliminary work done by the
authors in \cite{CohlCostasSantosGe,CohlCostasSantos20c}. In order to study the symmetry properties of the terminating basic hypergeometric functions which appear in the series representations of the Askey--Wilson polynomials, a detailed parametric connection between them was provided in 
\cite[Corollary 3]{CohlCostasSantosGe}. However, there were some typographical errors in that result and some representations which arise by inversion were inadvertently left off. An attempt to remedy this was executed in \cite{CohlCostasSantos20c} (we also missed some of the connections between classes of 4-parameter symmetric interchange transformations in \cite[\S3.3]{CohlCostasSantosGe}, a complete description is now given in Appendix \ref{fullcol} below). 

Further continuation 
of our study of the mapping properties of these functions was made clear by previous and our work on the group theoretic description of the transformation properties of these functions (see e.g., \cite{Lievensetal2007,VanderJeugtRao1999} and Propositions \ref{firstS6}, \ref{secS6} below). This work in this present paper provides a framework for future work on the symmetry analysis of terminating basic hypergeometric functions which is more complicated than that for the nonterminating case \cite{VanderJeugtpriv2020} and that it is not surprising that the classes of terminating basic hypergeometric functions are not connected by the known nonterminating transformations (see Figures \ref{figinv}, \ref{figinvinv}, \ref{figinvinvinv} below). In this paper, for the first time, we present the full symmetry structure of the terminating ${}_8W_7$ representations for the Askey--Wilson polynomials and a detailed connection with the terminating balanced ${}_4\phi_3$ representations.

\section{Preliminaries}
\label{Preliminaries}
We adopt the following set 
notations:~$\N_0:=\{0\}\cup\N=\{0, 1, 2, ...\}$, and we use the sets $\Z$, $\R$, $\CC$ which represent 
the integers, real numbers  and complex numbers respectively, 
$\CCast:=\CC\setminus\{0\}$,
and $\CCdag:=\CCast\setminus
\{z\in\CC: |z|=1\}$.
We also adopt the following notation and conventions.
Let ${\bf a}:=\{a_1, a_2, a_3, a_4\}$, $b, a_k\in\CC$, $k=1, 2, 3, 4$. 
Define
${\bf a}+b:=\{a_1+b,a_2+b,a_3+b,a_4+b\}$,
$a_{12}:=a_1a_2$,
$a_{13}:=a_1a_3$,
$a_{23}:=a_2a_3$,
$a_{123}:=a_1a_2a_3$,
$a_{1234}:=a_1a_2a_3a_4$, etc. 
%Consider a sequence of complex numbers $\{a_k\}$, $k\in\N_0$. 
%We adopt the following sum and product notations
%\[
%a_1+\cdots+a_n:=\sum_{k=1}^n a_k,\quad
%a_1\cdots %a_n:=\prod_{k=1}^n a_k.
%\]
%Furthermore, let $s,r\in\N_0$, with $s<r$. 
%Then 
Throughout the paper, we assume that the empty sum 
vanishes and the 
empty product 
is unity.
%, namely
%\[
%\sum_{k=r}^s a_k=0, 
%\quad \prod_{k=r}^s %a_k=1.
%\]

\begin{defn} \label{def:2.1}
Throughout this paper we adopt the following conventions for succinctly 
writing elements of lists. To indicate sequential positive and negative 
elements, we write
\[
\pm a:=\{a,-a\}.
\]
\noindent We also adopt an analogous notation
\[
\expe^{\pm i\theta}:=\{\expe^{i\theta},\expe^{-i\theta}\}.
\]
%If $\pm$ appears in an expression, but not in 
%a list, it is to be treated as normal.
\noindent In the same vein, consider the numbers $f_s\in\CC$ with $s\in{\mathcal S}\subset \N$,
with ${\mathcal S}$ finite.
Then, the notation
$\{f_s\}$
represents the set of all complex numbers $f_s$ such that 
$s\in\SSS$.
Furthermore, consider some $p\in\SSS$, then the notation
$\{f_s\}_{s\ne p}$ represents the sequence of all complex numbers
$f_s$ such that $s\in\SSS\!\setminus\!\{p\}$.
In addition, for the empty list, $n=0$, we take
\[
\{a_1,{...},a_n\}:=\emptyset.
\]
\end{defn}

Consider $q\in\CCdag$.
Define the sets 
$\Omega_q^n:=\{q^{-k}:n,k\in\N_0,~0\le k\le n-1\}$,
$\Omega_q:=\Omega_q^\infty=\{q^{-k}:k\in\N_0\}$.
In order to obtain our derived identities, we rely on properties 
of the $q$-Pochhammer symbol ($q$-shifted factorial). 
For any $n\in \N_0$, $a,q \in \CC$, 
the 
%Pochhammer symbol, and 
$q$-Pochhammer symbol is defined as
%\vspace{12pt}
\begin{eqnarray}
%&&\hspace{-3.3cm}%(a)_0:=1, \quad 
%(a)_n:=(a)(a+1)\cdots(a+n-1),\quad n\in\N_0,\nonumber\\
&&\hspace{-5.3cm}\label{poch.id:1} %(a;q)_0:=1, \quad 
(a;q)_n:=(1-a)(1-aq)\cdots(1-aq^{n-1}),\quad n\in\N_0.
\end{eqnarray}
One may also define
\begin{equation}
(a;q)_\infty:=\prod_{n=0}^\infty (1-aq^{n}), \label{poch.id:2}
\end{equation}
where $|q|<1$.
%%Furthermore, define for all $a,b\in\CC$, 
%%\[
%%(a)_b:=\dfrac{\Gamma(a+b)}{\Gamma(a)},
%%%=\dfrac{(a)_\infty}{(a+b)_\infty},
%%\]
%%where $a+b\not\in-\N_0$, 
%Furthermore, define
%\[
%(a;q)_b:=\frac{(a;q)_\infty}{(a q^b;q)_\infty}.
%\]
%where $a q^b\not \in \Omega_q$.
We will also use the common notational product conventions
\begin{eqnarray}
%&&\hspace{-8cm}(a_1,...,a_k)_b:=(a_1)_b\cdots(a_k)_b,\nonumber\\
&&\hspace{-8cm}(a_1,...,a_k;q)_b:=(a_1;q)_b\cdots(a_k;q)_b.\nonumber
\end{eqnarray}
%$q$-Pochhammer symbols
%are used in $q$-special functions.
%We define the $q$-factorial as 
%\cite[(1.2.44)]{GaspRah}
%\[
%[0]_q!:=1, \ [n]_q!:=[1]_q[2]_q\cdots [n]_q, \quad n\in\N,
%\]
%where the $q$-number is defined as \cite[(1.8.1)]{Koekoeketal}
%\[
%[z]_q:=\frac{1-q^z}{1-q}, \quad z\in \mathbb C.
%\]
%Note that $[n]_q!=(q;q)_n / (1-q)^n$, $n\in\N_0$.

The following properties for the $q$-Pochhammer 
symbol can be found in Koekoek et al. 
\cite[(1.8.7), (1.8.10-11), (1.8.14), (1.8.19), (1.8.21-22)]{Koekoeketal}, 
namely for appropriate values of $q,a\in\CCast$ and $n,k\in\mathbb N_0$:
\begin{eqnarray}
\label{poch.id:3} &&\hspace{-8.5cm}
%(a;q)_n=q^{\binom n 2}(-a)^n(a^{-1};q^{-1})_n,
(a;q^{-1})_n=(a^{-1};q)_n(-a)^nq^{-\binom{n}{2}}.
%\\[2mm] 
%\label{poch.id:4}
%&&\hspace{-6.5cm}(a;q)_{n+k}=(a;q)_k(aq^k;q)_n 
%= (a;q)_n(aq^n;q)_k,\\[2mm] 
%\label{poch.id:5}&&\hspace{-6.5cm} (a;q)_n=(q^{1-n}/a;q)_n(-a)^nq^{\binom{n}{2}},\\[2mm]
%\label{poch.id:6}&&\hspace{-6.5cm}(aq^{-n};q)_{k}=q^{-nk}
%\frac{(q/a;q)_n}{(q^{1-k}/a;q)_n}(a;q)_k,\\[2mm]
%\label{poch.id:8}&&\hspace{-6.5cm}(a^2;q^2)_n=(\pm a;q)_n,\\[2mm]
%\label{poch.id:7}&&\hspace{-6.5cm}(a;q)_{2n}=(a,aq;q^2)_n=(\pm\sqrt{a},\pm\sqrt{qa};q)_n.
\end{eqnarray}
%\noindent Observe that by using (\ref{poch.id:1}) and 
%(\ref{poch.id:8}), one obtains
%\begin{eqnarray}
%&&\hspace{-6.5cm}(aq^n;q)_n=\frac{(\pm 
%\sqrt{a},\pm \sqrt{aq};q)_n}{(a;q)_n},\quad 
%a\not\in\Omega_q^n. 
%\label{poch.id:9}
%\end{eqnarray}

The basic hypergeometric series, which we 
will often use, is defined for
$q,z\in\CCast$, such that $|q|,|z|<1$, $s,r\in\N_0$, 
$b_j\not\in\Omega_q$, $j=1,...,s$, as
\cite[(1.10.1)]{Koekoeketal}
\begin{equation}
\qhyp{r}{s}{a_1,...,a_r}
{b_1,...,b_s}
{q,z}
:=\sum_{k=0}^\infty
\frac{(a_1,...,a_r;q)_k}
{(q,b_1,...,b_s;q)_k}
\left((-1)^kq^{\binom k2}\right)^{1+s-r}
z^k.
\label{2.11}
\end{equation}
\noindent Note that we refer to a basic hypergeometric
series as { $\ell$-balanced} if
$q^\ell a_1\cdots a_r=b_1\cdots b_s$, 
and { balanced} 
\boro{(Saalsch\"utzian)} if $\ell=1$.
A basic hypergeometric series ${}_{r+1}\phi_r$ is { well-poised} if 
the parameters satisfy the relations
\[
qa_1=b_1a_2=b_2a_3=\cdots=b_ra_{r+1}.
\]
\noindent It is {very-well poised} if in addition, 
$\{a_2,a_3\}=\pm q\sqrt{a_1}$.

Similarly for terminating basic hypergeometric series 
which appear in basic hypergeometric orthogonal
polynomials, one has
\begin{equation}
\qhyp{r}{s}{q^{-n},a_1,...,a_{r-1}}
{b_1,...,b_s}{q,z}:=\sum_{k=0}^n
\frac{(q^{-n},a_1,...,a_{r-1};q)_k}{(q,b_1,...,b_s;q)_k}
\left((-1)^kq^{\binom k2}\right)^{1+s-r}z^k,
\label{2.12}
\end{equation}
where $b_j\not\in\Omega_q^n$, $j=1,...,s$.
%Note that \cite[p.~15]{Koekoeketal}
%\[
%\lim_{q\uparrow 1^{-}}
%\qhyp{r}{s}{q^{a_1},...,q^{a_r}}{q^{b_1},...,q^{b_s}}{q,(q-1)^{1+s-r}z}
%\nonumber\\ &&
%=\Fhyp{r}{s}{a_1,...,a_r}{b_1,...,b_s}{z}
%:=\sum_{k=0}^\infty 
%\frac{(a_1,...,a_r)_k}{(b_1,...,b_s)_k}\frac{z^k}{k!},
%\end{eqnarray}
%\]
%where ${}_rF_s$ is the generalized hypergeometric series %\cite[Chapter 16]{NIST:DLMF}.
Define the very-well poised 
basic hypergeometric series
${}_{r+1}W_r$ \cite[(2.1.11)]{GaspRah}
\begin{equation}
\label{rpWr}
{}_{r+1}W_r(b;a_4, ..., a_{r+1};q,z)
:=\qhyp{r+1}{r}{b,\pm q\sqrt{b},a_4, ..., a_{r+1}}
{\pm \sqrt{b},\frac{qb}{a_4}, ..., \frac{qb}{a_{r+1}}}{q,z},
\end{equation} 
where $\sqrt{b},\frac{qb}{a_4}, ..., \frac{qb}{a_{r+1}}\not\in\Omega_q$. 
When the very-well poised basic hypergeometric
series is terminating, then one has
\begin{equation}
{}_{r+1}W_r\left(b;q^{-n},a_5, ..., a_{r+1};q,z\right)
=\qhyp{r+1}{r}{b,\pm q\sqrt{b},q^{-n},a_5, ..., a_{r+1}}
{\pm \sqrt{b},q^{n+1}b,\frac{qb}{a_5}, ..., \frac{qb}{a_{r+1}}}
{q,z},
\label{eq:2.13}
\end{equation}
where $\sqrt{b},\frac{qb}{a_5}, ..., \frac{qb}{a_{r+1}}\not\in\Omega_q^n\cup\{0\}$.
The Askey--Wilson polynomials are intimately connected with 
the terminating very-well poised ${}_8W_7$, which is
given by
\begin{equation}
\label{VWP87}
{}_{8}W_7(b;q^{-n}\!,c,d,e,f;q,z)
=\qhyp{8}{7}{b,\pm q\sqrt{b},q^{-n},c,d,e,f}
{\pm \sqrt{b},q^{n+1}b,\frac{qb}{c},\frac{qb}{d},\frac{qb}{e},\frac{qb}{f}}{q,z},
\end{equation}
where $\sqrt{b},\frac{qb}{c},\frac{qb}{d},\frac{qb}{e},\frac{qb}{f}\not\in\Omega_q^n\cup\{0\}$.

\medskip

%We will often use (frequently without mentioning) the 
%following limit transition 
%formulas which can be found in \cite[(1.10.3-5)]{Koekoeketal}
%\begin{eqnarray}
%\label{limit1} \lim_{\lambda \to \infty} \qhyp{r}{s}{a_1,...,a_{r-1},\lambda a_r}
%{b_1,...,b_s}{q,\frac{z}{\lambda}}=\qhyp{r-1}{s}{a_1,...,a_{r-1}}{b_1,...,b_s}{q,a_rz}, \\
%\label{limit2} \lim_{\lambda \to \infty} \qhyp{r}{s}{a_1,...,a_{r}}{b_1,...,b_{s-1},\lambda b_s}
%{q,\lambda z}=\qhyp{r}{s-1}{a_1,...,a_{r}}{b_1,...,b_{s-1}}{q,\frac{z}{b_s}},\\
%\label{limit3} \lim_{\lambda \to \infty} \qhyp{r}{s}{a_1,...,a_{r-1},\lambda a_r}
%{b_1,...,b_{s-1},\lambda b_s}{q,z}=\qhyp{r-1}{s-1}{a_1,...,a_{r-1}}{b_1,...,b_{s-1}}{q,\frac{a_r}{b_s}z}.
%\end{eqnarray}

%\noro{

{
The following notation %${}_{r+1}\phi_s^m$
$\nqphyp{r+1}{s}{m}$,  $m\in\mathbb Z$
(originally due to van de Bult \& Rains
\cite[p.~4]{vandeBultRains09}),  
for basic hypergeometric series with
zero parameter entries.
Consider $p\in\mathbb N_0$. Then define
\begin{equation}\label{topzero} 
\nlqphyp{r+1}{s}{-p}\!
%{}_{r+1}\phi_s^{-p}
\left(\!\begin{array}{c}a_1,\ldots,a_{r+1} \\
b_1,\ldots,b_s\end{array};q,z
\right)
:=
%\qhyp{r+p+1}{s}{a_1,\ldots,
%a_{r+1},\overbrace{0,...,0}^{p}}
%{b_1,\ldots,b_s}{q,z},
\rphisx{r+p+1}{s}{a_1,a_2,\ldots,a_{r+1},\overbrace{0,\ldots,0}^{p} \\ b_1,b_2,\ldots,b_s}{z},
\end{equation}
\begin{equation}\label{botzero}
\nlqphyp{r+1}{s}{\,p}\!
%{}_{r+1}\phi_s^{\,p}
\left(\!\begin{array}{c}a_1,\ldots,a_{r+1} \\
b_1,\ldots,b_s\end{array};q,z
\right)
:=
%\qhyp{r+1}{s+p}
%{a_1,\ldots,a_{r+1}}
%{b_1,\ldots,b_s,
%\underbrace{0,...,0}_{p}}{q,z},
\rphisx{r+1}{s+p}{a_1,a_2,\ldots,a_{r+1} \\ b_1,b_2,\ldots,b_s, \underbrace{0,\ldots,0}_p}{z},
\end{equation}
where $b_1,\ldots,b_s\not
\in\Omega_q\cup\{0\}$, and
$
\nlqphyp{r+1}{s}{0}
={}_{r+1}\phi_{s}
.$
The terminating basic hypergeometric series
$\nlqphyp{r+1}{s}{m}
(q^{-n},{\bf a};{\bf b};q,z)$, for some $n\in\N_0$, ${\bf a}:=\{a_1,\ldots,a_{r}\}$,
${\bf b}:=\{b_1,\ldots,b_s\}$, is well-defined for
all $r,s\in\N_0$, $m\in\Z$. 
%Note that we may move interchangeably between the
%van de Bult \& Rains notation and the alternative
%notation with vanishing numerator and denominator parameters
%which are used on the right-hand sides of \eqref{topzero} and \eqref{botzero}.
}
In \cite[Exercise 1.4ii]{GaspRah} one finds the inversion formula for
terminating basic hypergeometric series.

\begin{thm}[Gasper and Rahman{'s} (2004) {Inversion Theorem}] \label{thm:2.2}
Let $m, n, k, r, s\in\N_0$, $a_k\in\CC$,  $1\le k\le r$,
$b_m\not \in\Omega^n_q$, 
$1\le m\le s$,
$q\in\CCdag$.
Then,
\begin{eqnarray}
&&\hspace{-0.7cm}\qhyp{r+1}{s}{q^{-n},a_1,...,a_r}{b_1,...,b_s}{q,z}
{=\frac{(a_1,...,a_r;q)_n}
{(b_1,...,b_s;q)_n}\left(\frac{z}{q}\right)^n\left((-1)^n
q^{\binom{n}{2}}\right)^{s-r-1}}
%=\frac{(a_1,...,a_r;q)_n}
%{(b_1,...,b_s;q)_n}\left(\frac{z}{q}\right)^n\left((-1)^n
%q^{\binom{n}{2}}\right)^{s-r-1} \nonumber \\
%&& \hspace{5.0cm} \times \sum_{k=0}^n \frac{\left(q^{-n},\frac{q^{1-n}}{b_1},...,\frac{q^{1-n}}{b_s};q\right)_k}
%{\left(q,\frac{q^{1-n}}{a_1},...,\frac{q^{1-n}}{a_r};q\right)_k} \left(
%\frac{q^{n+1}}{z}
%\frac{b_1 \cdots b_s}{a_1 %\cdots a_r}
%\right)^k
\nonumber\\
&&\hspace{5.0cm}{\times\qphyp{s+1}{r}{s-r}
{q^{-n},\frac{q^{1-n}}{b_1},...,\frac{q^{1-n}}{b_s}}
{\frac{q^{1-n}}{a_1},...,\frac{q^{1-n}}{a_r}}
{q,
\frac{q^{n+1}}{z}
\frac{b_1 \cdots b_s}{a_1 \cdots a_r}
}}.
\label{inversion}
\end{eqnarray}
\end{thm}
\begin{cor}\label{cor:2.4}
Let $n,r\in\N_0$, $q\in\CCdag$, 
$a_k, b_k\not\in\Omega^n_q\cup\{0\}$,
$1\le k\le r$.
Then,
\begin{eqnarray}
&&\hspace{-1cm}\qhyp{r+1}{r}{q^{-n},a_1, ..., a_r}{b_1, ..., b_r}{q,z}\nonumber\\\label{cor:2.4:r1}
&&\hspace{1cm}=
q^{-\binom{n}{2}}
\left(\!-\frac{z}{q}\right)^{\!\!n}
\frac{(a_1, ..., a_r;q)_n}
{(b_1, ..., b_r;q)_n}
\qhyp{r+1}{r}{q^{-n},
\frac{q^{1-n}}{b_1}, ..., 
\frac{q^{1-n}}{b_r}}
{\frac{q^{1-n}}{a_1}, ..., 
\frac{q^{1-n}}{a_r}}{q,
\frac{q^{n+1}}{z}\frac{b_1\cdots b_r}
{a_1\cdots a_r}}.
\end{eqnarray}
\end{cor}
\begin{proof}
Take $r=s$, in \eqref{inversion} and using the definition \eqref{2.11}
completes the proof.
%$p=0$ in 
%\eqref{ivg1}, \eqref{ivg2}, \eqref{ivg3}, \eqref{ivg4} which all coincide
%with no zero entries, which completes
%the proof.
\end{proof}
\noindent Note that in Corollary \ref{cor:2.4}
if the terminating basic hypergeometric
series on the left-hand side is balanced
then the argument of the terminating
basic hypergeometric series on the right-hand side is~$q^2/z$.
%}

\medskip
Applying Corollary \ref{cor:2.4} to the definition
of ${}_{r+1}W_r$, we obtain the following result for a terminating very-well 
poised basic hypergeometric series ${}_{r+1}W_r$.%}
\begin{cor} \label{cor:2.5}
Let 
$n\in\N_0$, $b,a_k,q,z\in\CCast$, $\sqrt{b},q^{n+1}b,\frac{q b}{a_k},%\frac{q^{-n}}{\sqrt{b}},
\frac{q^{1-n}}{b},\frac{q^{1-n}}{a_k}\not\in\Omega_q^n$, $k=5, ..., r+1$. 
Then, one has the following transformation formula for a very-well poised terminating 
basic hypergeometric series:
\begin{eqnarray}
&&\hspace{-0.4cm}{}_{r+1}W_r\left(b;q^{-n},a_5, ..., a_{r+1};q,z\right)=
q^{-\binom{n}{2}}\left(\frac{-z}{q}\right)^n
\frac{(\pm q\sqrt{b},b,a_5, ..., a_{r+1};q)_n}
{\left(\pm \sqrt{b},q^{n+1}b,\frac{qb}{a_5}, ..., \frac{qb}{a_{r+1}};q\right)_n}\nonumber\\
&&\hspace{4.5cm}\times{}_{r+1}W_r\left(\frac{q^{-2n}}{b};q^{-n},
\frac{q^{-n}a_5}{b}, ..., \frac{q^{-n}a_{r+1}}{b};q,\frac{q^{2n+r-3}b^{r-3}}{(a_5\cdots a_{r+1})^2 z}\right).
\end{eqnarray}
\end{cor}
\begin{proof}
Use Corollary \ref{cor:2.4} and \eqref{eq:2.13}.
\end{proof}

An interesting and useful consequence of this formula is the $r=7$ special case,
%\begin{equation}
%\label{4W3inverse}
%{}_4W_3\left(b;q^{-n};q,z\right)=q^{-\binom{n}{2}}\left(\frac{-z}{q}\right)^n
%\frac{(\pm q\sqrt{b},b;q)_n}
%{(\pm\sqrt{b},q^{n+1}b;q)_n}
%{}_4W_3
%\left(\frac{q^{-2n}}{b};q^{-n};q,\frac{q^{2n}}{z}\right),
%\end{equation}
%%\noro{
%and the $r=7$ specialization 
\begin{eqnarray}
&&{}_8W_7\left(b;q^{-n},c,d,e,f;q,z\right)
=q^{-\binom{n}{2}} \left(\frac{-z}{q}\right)^n
\frac{\left(\pm q\sqrt{b},b,c,d,e,f;q\right)_n}
{\left(\pm\sqrt{b},q^{n+1}b,\frac{qb}{c},\frac{qb}{d},\frac{qb}{e},\frac{qb}{f};q\right)_n}\nonumber\\
&&\hspace{4.5cm}\times
{}_8W_7\left(\frac{q^{-2n}}{b};q^{-n},
\frac{q^{-n}c}{b},
\frac{q^{-n}d}{b},
\frac{q^{-n}e}{b},
\frac{q^{-n}f}{b};
q,\frac{q^{2n+4}b^4}{z(cdef)^2}
\right).
\label{inv87}
\end{eqnarray}
%\noro{[RCS: I changed a bit the next sentence.]}\\
\noindent \boro{Note that %in the case 
when one obtains an ${}_8W_7$ from a balanced}
${}_4\phi_3$ using \eqref{WatqWhipp},
then $q^{2n+4}b^4/(z(cdef)^2)=z$.

\medskip

\noindent We will obtain new transformations for basic hypergeometric orthogonal
polynomials by taking advantage of the following remark.

\begin{rem} \label{rem:2.7}
Since $x=\cos\theta$ is an even function of $\theta$, 
all polynomials in $\cos\theta$ will be invariant under the 
map $\theta\mapsto-\theta$.
\end{rem}

\begin{rem} \label{rem:2.8}
Observe in the following discussion we will often be referring
to a collection of constants $a,b,c,d,e,f$. In such cases, which will be
clear from context, then the constant $e$ should not be confused
with Euler's number $\expe$, the base of the natural logarithm, i.e.,
$\log\expe=1$. Observe the different
(roman) typography for Euler's number.
\end{rem}

\section{The Askey--Wilson polynomials}
\label{askeywilson}

Define the sets ${\bf 4}:=\{1,2,3,4\}$, 
${\bf a}:=(a_1,a_2,a_3,a_4)$, $a_k\in\CCast$, $k\in{\bf 4}$, and
$x=\cos\theta\in[-1,1]$.
{The Askey--Wilson polynomials $p_n(x;{\bf a}|q)$ are a family of polynomials symmetric 
in the four free parameters $a_1$, $a_2$, $a_3$ and $a_4$.}
These polynomials have a
long and in-depth history and their properties have been studied in detail. 
The basic hypergeometric series representation of the Askey--Wilson polynomials fall into four main categories:~(1) terminating ${}_4\phi_3$ representations; (2) terminating ${}_8W_7$ representations; (3)~nonterminating ${}_8W_7$ representations; and (4) nonterminating ${}_4\phi_3$ representations.
One may obtain the alternative 
nonterminating representations of the Askey--Wilson polynomials using 
\cite[(2.10.7)]{GaspRah}
and
\cite[17.9.16]{NIST:DLMF}.
However, these nonterminating representations will not be further 
discussed in this paper.

%These series representations are a key to finding many properties of 
%and formulas for the Askey--Wilson polynomials. 
%Different series representations are useful for obtaining different 
%properties and formulas for these polynomials. So it is very useful 
%to have at hand an exhaustive list.
%The discussion contained in this section is an attempt to summarize, 
%in an in-depth manner, an exhaustive description of the representation 
%and transformation properties of the {terminating}
%${}_4\phi_3$ and ${}_8W_7$ basic hypergeometric representations of 
%the Askey--Wilson polynomials. 
%

%Our main goal in this section is to obtain an exhaustive list of 
%terminating basic hypergeometric series representations for the 
%Askey--Wilson polynomials.
%There is a directed graph to each and every series representation 
%which is useful to understand the origin and behavior of these %representations.
%The remaining sections below which contain the terminating basic 
%hypergeometric 
%representations and transformations for the symmetric subfamilies of the 
%Askey--Wilson polynomials all follow from the results presented in this section.

\subsection{The Askey--Wilson polynomial terminating series representations}

First we present the terminating series representations of the Askey--Wilson polynomials. They are given in terms of terminating balanced ${}_4\phi_3$ and terminating  very-well-poised ${}_8W_7$ basic hypergeometric series. This result was presented in \cite[Theorem 3]{CohlCostasSantosGe}.
The symmetric structure of the mapping properties of the utilized basic hypergeometric functions which appear in this theorem are the essential ingredients for the remainder of the paper.

\begin{thm} \label{thm:3.1}
Let $n\in\N_0$, $p,s,r,t,u\in {\bf 4}$, $p,r,t,u$ distinct and fixed, 
%${\bf a}:=(a_1,a_2,a_3,a_4)$, 
%$x=\cos\theta\in[-1,1]$, 
$q\in\CCdag$.
Then, the Askey--Wilson polynomials have the following terminating basic hypergeometric series representations given by:
\begin{eqnarray}
\hspace{-0.50cm}
\label{aw:def1} p_n(x;{\bf a}|q) &&:= a_p^{-n} \left(\{a_{ps}\}_{s \neq p};q\right)_n 
\qhyp43{q^{-n},q^{n-1}a_{1234}, a_p\expe^{\pm i\theta}}{\{a_{ps}\}_{s \neq p}}{q,q} \\
\label{aw:def2} &&=q^{-\binom{n}{2}} (-a_p)^{-n} 
\frac{\left(\frac{a_{1234}}{q};q\right)_{2n}
\left(a_p \expe^{\pm i\theta};q\right)_n}
{\left(\frac{a_{1234}}{q};q\right)_n}
\qhyp43{q^{-n}, \left\{\frac{q^{1-n}}{a_{ps}}\right\}_{s \neq p}}
{\frac{q^{2-2n}}{a_{1234}}, \frac{q^{1-n}\expe^{\pm i\theta}}{a_p}}{q,q} \\
\label{aw:def3} &&=\expe^{in\theta} \left(a_{pr},a_t\expe^{-i\theta},a_u\expe^{-i\theta};q\right)_n 
\qhyp43{q^{-n}, a_p\expe^{i\theta}, a_r\expe^{i\theta}, \frac{q^{1-n}}{a_{tu}}}{a_{pr}, \frac{q^{1-n}\expe^{i\theta}}{a_t}, 
\frac{q^{1-n}\expe^{i\theta}}{a_u}}{q,q} \\
%&&=a_p^{-n} (a_{ps},a_{pt};q)_n \frac{(a_r\expe^{\pm %i\theta};q)_n}{(\frac{a_r}{a_p};q)_n} \nonumber \\ 
%\label{aw:def5} && \hspace{0.7 cm} \times \qhyp87{q^{-n}, \frac{q^{-n}a_p}{a_r}, 
%\pm q^{1-\frac{n}{2}}\left({\frac{a_p}{a_r}}\right)^\frac12\!, 
%\frac{q^{1-n}}{a_{rs}}, \frac{q^{1-n}}{a_{rt}}, a_p\expe^{\pm i\theta}}
%{\pm q^{-\frac{n}{2}}\left({\frac{a_p}{a_r}}\right)^{\frac12}\!, a_{ps}, a_{pt}, \frac{qa_p}{a_r}, 
%\frac{q^{1-n}\expe^{\pm i\theta}}{a_r}}{q,q^n a_{st}}.
&&=\expe^{in\theta}
\dfrac{\left(\frac{a_{1234}}{q};q\right)_{2n}\left(\left\{a_s\expe^{-i\theta}\right\}_{s\ne p}
,\frac{a_{1234}\,\expe^{-i\theta}}{q a_p};q\right)_{n}}
{\left(\frac{a_{1234}}{q};q\right)_{n}
\left(\frac{a_{1234}\,\expe^{-i\theta}}{q a_p};q\right)_{2n}}
\nonumber \\
&&\label{aw:def6}
\hspace{3cm}\times\,{}_8W_7\left(
\frac{q^{1-2n}a_p\expe^{i\theta}}{a_{1234}};q^{-n},
\left\{\frac{q^{1-n}a_{ps}}{a_{1234}}\right\}_{s\ne p},
a_p\expe^{i\theta};q,
\frac{q\expe^{i\theta}}{a_p}\right)\\
&&\label{aw:def7}
=\expe^{in\theta}
\dfrac{\left(a_p\expe^{-i\theta},\{\frac{a_{1234}}{a_{ps}}\}_{s\ne p};q\right)_n}
{\left(\frac{a_{1234}\,\expe^{i\theta}}{a_p};q\right)_{n}}
%}\nonumber \\&&\noro{\hspace{3cm}\times\,
{}_8W_7\left(
\frac{a_{1234}\expe^{i\theta}}{qa_p};q^{-n},
\{a_s\expe^{i\theta}\}_{s\ne p},q^{n-1}a_{1234};q,
\frac{q\expe^{-i\theta}}{a_p}\right)\\
&&\label{aw:def5}=a_p^{-n}\dfrac{\left(a_{pt},a_{pu},a_r\expe^{\pm i\theta};q\right)_n}{\left(\frac{a_r}{a_p};q\right)_n} 
\,{}_8W_7\left(
\frac{q^{-n}a_p}{a_r};q^{-n},\!
\frac{q^{1-n}}{a_{rt}}, \frac{q^{1-n}}{a_{ru}}, a_p\expe^{\pm i\theta};q,q^n a_{tu}\right)\\
%\qhyp87{q^{-n}, 
%\pm q^{1-\frac{n}{2}}
%\sqrt{\frac{a_p}{a_r}}, 
%\frac{q^{-n}a_p}{a_r}, 
%\frac{q^{1-n}}{a_{rt}}, \frac{q^{1-n}}{a_{ru}}, a_p\expe^{\pm i\theta}}
%{\pm q^{-\frac{n}{2}}
%\sqrt{\frac{a_p}{a_r}}, 
%a_{pt}, a_{pu}, \frac{qa_p}{a_r}, 
%\frac{q^{1-n}\expe^{\pm i\theta}}{a_r}}{q,q^n a_{tu}}\\
\label{aw:def4} &&=\expe^{in\theta} \frac{\left(\{a_s\expe^{-i\theta}\};q\right)_n}{\left(\expe^{-2i\theta};q\right)_n} 
\,{}_8W_7\left(
q^{-n}\expe^{2i\theta};q^{-n},
\{a_s\expe^{i\theta}\};
q,\frac{q^{2-n}}{a_{1234}}
\right).
%\qhyp87{q^{-n},
%\pm q^{1-\frac{n}{2}}\expe^{i\theta},
%q^{-n}\expe^{2i\theta}, 
%\{a_s\expe^{i\theta}\}}
%{\pm q^{-\frac{n}{2}}\expe^{i\theta}, q\expe^{2i\theta}, \left\{\frac{q^{1-n}\expe^{i\theta}}{a_s}\right\}}{q,\frac{q^{2-n}}{a_{1234}}}.
\end{eqnarray}
\label{AWthm}
\end{thm}

\begin{proof}
See the proof of \cite[Theorem 3]{CohlCostasSantosGe}.
%The standard definition of the Askey--Wilson polynomials \eqref{aw:def1} is found in  many places including~\cite[(14.1.1)]{Koekoeketal}. The representation \eqref{aw:def2} can be derived by applying \eqref{inversion} to \eqref{aw:def1}. It also follows by using~(\cite[second equality in (17.9.14)]{NIST:DLMF}  with $a=a_{1234}q^{n-1}$, $\{b,c\}=a_p\expe^{\pm i\theta}$, $\{d,e,f\}=\{a_{ps}\}_{s\ne p}$. One can obtain \eqref{aw:def3} by starting with \eqref{aw:def1} and $p\leftrightarrow u$ using \cite[second equality in (17.9.14)]{NIST:DLMF}   with $\{a,b\}=a_u\expe^{\mp i\theta}$, $c=a_{1234}q^{n-1}$, $\{d,e,f\}=\{a_{us}\}_{s\ne u}$, and \cite[(1.8.14)]{Koekoeketal}. The representation \eqref{aw:def6} follows by from \eqref{aw:def7} by using the inversion formula \eqref{inv87}. The representation \eqref{aw:def7} follows from \eqref{aw:def1} with $p\leftrightarrow r$, using Watson's $q$-analogue of Whipple's theorem \eqref{WatqWhipp} and mapping $\theta\mapsto-\theta$ The representation \eqref{aw:def5} follows by using \cite[(III.19)]{GaspRah} directly with \eqref{aw:def1}. The representation \eqref{aw:def4} follows from \eqref{aw:def1}  and (\cite[(III.15), (III.19)]{GaspRah}, see also \cite[\S 14.1]{KoornwinderKLSadd}. This completes the proof.
\end{proof}

\begin{rem} \label{rem:3.2}
Please note the following symmetry properties of Theorem \ref{thm:3.1}. When inversion (Corollary \ref{cor:2.4}) is applied to \eqref{aw:def1}
one obtains \eqref{aw:def2}, and when one 
applies it to \eqref{aw:def3}, one obtains
the same formula back with $\theta\mapsto-\theta$ and $\{r,s\}\leftrightarrow\{t,u\}$.
Applying \eqref{inversion} to \eqref{aw:def3},
\eqref{aw:def6}, \eqref{aw:def7}, \eqref{aw:def4} simply takes 
$\theta \mapsto -\theta$, and applying it to \eqref{aw:def5} interchanges $a_p$ and $a_r$. 
Mapping $\theta \mapsto -\theta$ may give additional representations, however those are omitted.
\end{rem}

\subsection{Terminating 4-parameter symmetric transformations}

\begin{cor}
\label{cor:3.3}
Let $n\in\N_0$, $b,c,d,e,f\in\CCast$,
$q\in\CCdag$.
Then, one has 
the following transformation
formulas for a terminating ${}_8W_7$
to a terminating ${}_8W_7$:
\begin{eqnarray}
&&\label{aw:def4to5}\hspace{-0.44cm}{}_8W_7\left(b;q^{-n},c,d,e,f;q,\frac{q^{n+2}b^2}{cdef}\right)\\
&&\hspace{-0.21cm}\label{cor3.5a.1}=\!q^{\binom{n}{2}}\!\!
\left(\frac{-q^2b^2}{cdef}\right)^{\!n}\!\!\!\!\dfrac{\left(qb,b, c, d, e, f;q\right)_n}
{(b;q)_{2n}\!
\left(\!\frac{qb}{c},\!\frac{qb}{d},\!\frac{qb}{e},\!\frac{qb}{f};q\!\right)_{\!n}} 
{}_8W_7\!\left(\frac{q^{-2n}}{b};
q^{-n},\!\frac{q^{-n}c}{b},\!
\frac{q^{-n}d}{b},\!\frac{q^{-n}e}{b},\!\frac{q^{-n}f}{b};
q,\frac{q^{n+2}b^2}{cdef}
\right)\\
&&\hspace{-0.21cm}\label{cor3.5a.2}=\frac{\left(\frac{qb}{ce},\frac{qb}{cf},qb,d;q\right)_n}
{\left(\frac{qb}{c},\frac{qb}{e},\frac{qb}{f},\frac{d}{c};q\right)_n}
{}_8W_7\left(\frac{q^{-n}c}{d};q^{-n},\frac{q^{-n}c}{b},
\frac{qb}{de},\frac{qb}{df},c;q,\frac{ef}{b}\right)\\
&&\hspace{-0.21cm}\label{cor3.5a.7}
=\frac{\left(\frac{qb}{de},\frac{qb}{df},\frac{qb}{ef},qb;q\right)_n}
{\left(\frac{qb}{def},\frac{qb}{d},\frac{qb}{e},\frac{qb}{f};q\right)_n}
{}_8W_7\left(\frac{q^{-n-1}def}{b};q^{-n},d,e,f,\frac{q^{-n-1}cdef}{b^2};q,\frac{q}{c}\right)
\\
&&\hspace{-0.21cm}\label{cor3.5a.7b}
=\frac{\left(
\frac{q^2b^2}{cdef},qb,d,e,f;q\right)_n}
{\left(\frac{def}{qb},\frac{qb}{c},\frac{qb}{d},\frac{qb}{e},\frac{qb}{f};q\right)_n}
{}_8W_7\left(\frac{q^{1-n}b}{def};q^{-n},\frac{q^{-n}c}{b},
\frac{qb}{de},\frac{qb}{df},\frac{qb}{ef};q,\frac{q}{c}\right)
\\[-0.0cm]
&&\hspace{-0.21cm}\label{cor3.5a.6}
=\frac{\left(\frac{q^2b^2}{cdef},{qb};q\right)_n}
{\left(\frac{qb}{c},\frac{q^2b^2}{def};q\right)_n}
{}_8W_7\left(\frac{qb^2}{def};q^{-n},\frac{qb}{de},
\frac{qb}{df},\frac{qb}{ef},c;q,\frac{q^{n+1}{b}}{c}\right)
\\
&&\hspace{-0.21cm}\label{cor3.5a.7c}
=
q^{\binom{n}{2}}
\left(-\frac{qb}{c}\right)^n\frac{\left(
\frac{qb^2}{def},
\frac{qb}{ef},\frac{qb}{de},\frac{qb}{df},qb,c;q\right)_n}
{\left(\frac{qb^2}{def};q\right)_{2n}\left(\frac{qb}{c},\frac{qb}{d},\frac{qb}{e},\frac{qb}{f};q\right)_n}\nonumber\\[-0.00cm]
&&\hspace{3cm}\times{}_8W_7\left(\frac{q^{-2n-1}def}{b^2};q^{-n},\frac{q^{-n}d}{b},
\frac{q^{-n}e}{b},\frac{q^{-n}f}{b},\frac{q^{-n-1}cdef}{b^2};q,\frac{q^{n+1}b}{c}\right).
\end{eqnarray}
\end{cor}

\begin{proof}
Start with Theorem \ref{thm:3.1} and set $\expe^{2i\theta}=q^nb$, $a_p=q^{-\frac{n}{2}}\frac{c}{\sqrt{b}}$,
$a_r=q^{-\frac{n}{2}}\frac{d}{\sqrt{b}}$,
$a_t=q^{-\frac{n}{2}}\frac{e}{\sqrt{b}}$,
$a_u=q^{-\frac{n}{2}}\frac{f}{\sqrt{b}}$
\boro{, setting $\theta\mapsto-\theta$ where
necessary.}
Then, multiply every formula by the factor
\[
{\sf A}_n(b,c,d,e,f|q):=
\frac{q^{2\binom{n}{2}}
(-1)^n(qb)^\frac{5n}{2}\left(qb;q\right)_n
}
{
(cdef)^n\left(\frac{qb}{c},\frac{qb}{d},
\frac{qb}{e},\frac{qb}{f};q\right)_n
}.
\]
\noindent With simplification, this completes the proof.
\end{proof}

The above corollary relates a 
terminating very-well-poised ${}_8W_7$ to 
six other representations of terminating
very-well-poised ${}_8W_7$s.
The following corollary which results from comparing the symmetric ${}_8W_7$ representation of the Askey--Wilson polynomials to the
${}_4\phi_3$ representations of the Askey--Wilson polynomials is directly 
connected to Watson's $q$-analog of Whipple's
theorem \cite[(17.9.15)]{NIST:DLMF}. However,
beyond the single representation which 
is usually displayed, we are able to 
extend this to a total of four 
representations of terminating
balanced ${}_4\phi_3$s.

%four different expressions for a terminating 
%balanced ${}_4\phi_3$
%and the second corollary relates a 
%terminating balanced ${}_4\phi_3$ to seven different expressions of a terminating
%very-well-poised ${}_8W_7$.

%Keep for future reference
%\begin{eqnarray}
%&&\hspace{-3.0cm}{}_8W_7\left(b;q^{-n},c,d,e,f;q,\frac{q^{n+2}b^2}{cdef}\right)\\
%&&\hspace{-0.21cm}\label{cor3.5a.3}=\frac{\left(\frac{qb}{cd},qb;q\right)_n}{\left(\frac{qb}{c},\frac{qb}{d};q\right)_n}\qhyp43{q^{-n},\frac{qb}{ef},c,d}{\frac{q^{-n}cd}{b},\frac{qb}{e},\frac{qb}{f}}{q,q}\\
%&&\hspace{-0.21cm}
%%\label{cor3.5a.4}=
%\left(\frac{qb}{ef}\right)^n
%\dfrac{\left(\frac{qb}{cd},qb,e, f; q\right)_n}{\left(\frac{qb}{c},\frac{qb}{d},\frac{qb}{e}, \frac{qb}{f};q\right)_n} %\nonumber\\&& \hspace{-13mm} \times 
%\qhyp43{q^{-n}, \frac{q^{-n}c}{b}, 
%\frac{q^{-n}d}{b}, \frac{qb}{ef}}{\frac{q^{-n}cd}b, \frac{q^{1-n}}{e}, \frac{q^{1-n}}{f}}{q,q}\\
%&&\hspace{-0.21cm}
%%\label{cor3.5a.5}
%=\frac{\left(\frac{q^2b^2}{cdef},qb,c;q\right)_n}
%{\left(\frac{qb}{d},\frac{qb}{e},\frac{qb}{f};q\right)_n}
%\qhyp43{q^{-n},\frac{qb}{cd},\frac{qb}{ce},\frac{qb}{cf}}
%{\frac{q^2b^2}{cdef},\frac{q^{1-n}}{c},\frac{qb}{c}}{q,q}\\
%&&\hspace{-0.21cm}
%%\label{cor3.5a.6b}
%=c^n\frac{\left(\frac{qb}{cd},\frac{qb}{ce},\frac{qb}{cf},qb;q\right)_n}
%{\left(\frac{qb}{c},\frac{qb}{d},\frac{qb}{e},\frac{qb}{f};q\right)_n}
%\qhyp43{q^{-n},\frac{q^{-n-1}cdef}{b^2},\frac{q^{-n}c}{b},c}{\frac{q^{-n}cd}{b},\frac{q^{-n}ce}{b},\frac{q^{-n}cf}{b}}{q,q}.
%\end{eqnarray}

\begin{cor} (Watson's $q$-analog of Whipple's theorem \cite[(17.9.15)]{NIST:DLMF}).
\label{WatqWhipp2}
Let $n\in\N_0$, $b,c,d,e,f\in\CCast$, $q\in\CCdag$.
%where $q^{1-n}abc=def$.
Then
\begin{eqnarray}
&&\hspace{-0.5cm}{\Whyp87{b}{q^{-n},c,d,e,f}{q,\frac{q^{n+2}b^2}{cdef}}=
\frac{(qb,\frac{qb}{ef};q)_n}{
(\frac{qb}{e},\frac{qb}{f};q)_n}
\qhyp43{q^{-n},\frac{qb}{cd},e,f}{{\frac{q^{-n} e f}{b}},\frac{qb}{c},\frac{qb}{d}}{q,q}}
\label{cor3.5a.3}
%\label{WatqWhippfirst}
\\
&&\hspace{4.5cm}{=
\left(\frac{qb}{cd}\right)^n
\frac{(\frac{qb}{ef},qb,c,d;q)_n}{(\frac{qb}{c},\frac{qb}{d},\frac{qb}{e},\frac{qb}{f};q)_n}
\qhyp43{q^{-n},\frac{q^{-n}e}{b},\frac{q^{-n}f}{b},\frac{qb}{cd}}
{\frac{q^{-n}ef}{b},\frac{q^{1-n}}{c},\frac{q^{1-n}}{d}}{q,q}}
\label{cor3.5a.4}
\\
&&\hspace{4.5cm}{=\frac{\left(\frac{q^2b^2}{cdef},qb,e;q\right)_n}
{\left(\frac{qb}{c},\frac{qb}{d},\frac{qb}{f};q\right)_n}
\qhyp43{q^{-n},\frac{qb}{ec},\frac{qb}{ed},\frac{qb}{ef}}
{\frac{q^2b^2}{cdef},\frac{q^{1-n}}{e},\frac{qb}{e}}{q,q}}
\label{cor3.5a.5}
\\
&&\hspace{4.5cm}{=e^n\frac{\left(\frac{qb}{ec},\frac{qb}{ed},\frac{qb}{ef},qb;q\right)_n}
{\left(\frac{qb}{c},\frac{qb}{d},\frac{qb}{e},\frac{qb}{f};q\right)_n}
\qhyp43{q^{-n},\frac{q^{-n-1}cdef}{b^2},\frac{q^{-n}e}{b},e}{\frac{q^{-n}ec}{b},\frac{q^{-n}ed}{b},\frac{q^{-n}ef}{b}}{q,q}}
\label{cor3.5a.6b}.
\end{eqnarray}
Note that the above terminating ${}_4\phi_3$s are balanced.
\end{cor}

\begin{proof}
Same as in the proof of Corollary \ref{cor:3.3}
except applying the transformation
to the terminating balanced ${}_4\phi_3$s
in Theorem \ref{thm:3.1}.
%Start with \cite[(17.9.15)]{NIST:DLMF} and then use Corollary \ref{cor:3.3} to rewrite
%the ${}_4\phi_3$ expressions. {Also we have used relabeling of parameters such as in
%Corollary \ref{cor3.8}.} 
This completes the proof.
\end{proof}

%\begin{rem} \label{rem:17}
%\sout{Note that Corollary \ref{WatqWhipp2} is essentially a rewritten subset form of 
%Corollary \ref{cor:3.3}.}
%\end{rem}

\begin{rem} \label{rem:3.4}
%Notice that in Corollary \ref{cor:3.3}, our order
%of the representations begins with the principal ${}_8W_7$ representation 
%in which the symmetry in the parameters
%$c,d,e,f$ is evident and ends with the representation 
%corresponding to the classical 
%${}_4\phi_3$ basic hypergeometric representation of the Askey--Wilson 
%polynomials \eqref{aw:def1}. On the other hand, in Theorem \ref{AWthm},
%we have reversed the order of the corresponding 
%representations. The reason why we have used the
%ordering as such is because the ${}_4\phi_3$ representation
%of the Askey--Wilson polynomials 
%\eqref{aw:def1} is historically first (see \cite[(1.8)]{IsmailWilson82} 
%and the memoir \cite[(1.15)]{AskeyWilson85}) and is certainly the most
%common, see e.g.,~\cite[(14.1.1)]{Koekoeketal}.
The Askey--Wilson polynomials are
symmetric in their four parameters, the ${}_8W_7$ 
representation in which this symmetry is evident
demonstrates this symmetry.
On the other hand, the 
polynomial nature of the 
Askey--Wilson polynomials 
is not clearly evident
from the ${}_8W_7$ representation. 
In the first ${}_4\phi_3$ representation, 
the polynomial nature of evident.
\end{rem}

%\noindent \moro{\bf [Michael Schlosser: There is no need for writing out Corollary 17
%if it is part of Corollary 10.
%It is better to split Corollary 10 in two
%Corollaries, one involving only ${}_8W_7$ series,
%the other having ${}_4\phi_3$ series on the %right-hand-side.]\\[0.15cm]
%\noindent [HSC: Do not reproduce equations in Remark 16, 
%but explain situation. After that, then
%remove Remark 17.]
%}

%\begin{rem}
%Observe that under the standard map \eqref{stand}, the ${}_4\phi_3$ \eqref{cor3.5a.4}
%maps to the representation \eqref{aw:def3} with %$\theta\mapsto-\theta$.
%\end{rem}

\subsection{Converse for Watson's $q$-analog of Whipple's theorem}

One important transformation for terminating basic hypergeometric series 
related to the Askey--Wilson polynomials is Watson's $q$-analog of Whipple's 
theorem \cite[(17.9.15)]{NIST:DLMF}.
This result 
relates a terminating balanced ${}_4\phi_3$ to a terminating very-well 
poised ${}_8W_7$ .
The following corollary, an extension of this 
theorem, is a direct
consequence of Corollary \ref{WatqWhipp2}. Both of the following results directly relate 
a terminating balanced ${}_4\phi_3$ to
a terminating very-well-poised ${}_8W_7$. The balancing condition for the terminating
${}_4\phi_3$ is $q^{1-n}abc=def$. 

{
\begin{cor}(Converse for Watson's q-analog of Whipple's theorem).
\label{WatqWhipp}
Let $n\in\N_0$, $a,b,c,d,e,f\in\CCast$, $q\in\CCdag$, 
such that $q^{1-n}abc=def$ (balancing condition for the terminating ${}_4\phi_3$).
Then
\begin{eqnarray}
%\hspace{-1.35cm}
\qhyp43{q^{-n},a,b,c}{d,e,f}{q,q}
&=&\frac{(\frac{f}{b},\frac{f}{c};q)_n}
{(\frac{f}{bc},f;q)_n}
\Whyp87{\frac{q^{-n}bc}{f}}{q^{-n},\frac{e}{a},\frac{d}{a},
%\frac{q^{1-n}bc}{de},
b,c}{q,\frac{qa}{f}}\label{cWqW:1}\\
&=&\frac{(\frac{ef}{bc},\frac{e}{a},b,c;q)_n}
{(\frac{ef}{abc},\frac{bc}{f},e,f;q)_n}
\Whyp87{\frac{q^{-n}f}{bc}}{q^{-n},\frac{q^{1-n}}{d},
\frac{q^{1-n}}{e},%\frac{de}{abc},
\frac{f}{b},\frac{f}{c}}{q,\frac{qa}{f}}\label{cWqW:2}\\
&=&c^n\frac{(\frac{d}{c},\frac{e}{c},\frac{f}{c},b;q)_n}
{(\frac{b}{c},d,e,f;q)_n}
\Whyp87{\frac{q^{-n}c}{b}}{q^{-n},\frac{d}{b},\frac{e}{b},
\frac{f}{b},c}{q,\frac{q}{a}}\label{cWqW:3}\\
&=&\frac{(\frac{e}{a},\frac{e}{b},\frac{e}{c};q)_n}
{(\frac{de}{abc},\frac{e}{d},e;q)_n}
\Whyp87{\frac{q^{-n}d}{e}}{q^{-n},\frac{q^{1-n}}{e},\frac{d}{a},
\frac{d}{b},\frac{d}{c}}{q,q^nf}.
\label{cWqW:4}
\end{eqnarray}
\end{cor}
}
{
\begin{proof}
Consider \eqref{cor3.5a.3}, then solving the
following set of algebraic equations
\begin{equation}
\left(A,B,C,D,E,\frac{q^{1-n}BC}{DE}\right)
=\left(\frac{qb}{ef},c,d,\frac{q^{-n}cd}{b},\frac{qb}{e},\frac{qb}{f}\right),
\end{equation}
gives the solution
\begin{equation}
(b,c,d,e,f)=\left(\frac{q^{-n}BC}{D},B,C,\frac{q^{1-n}BC}{DE},\frac{F}{A}\right).
\end{equation}
Now make these replacements in 
\eqref{aw:def4to5}--\eqref{cor3.5a.7c}, and solving
for the ${}_4\phi_3$ in \eqref{cor3.5a.3},
while replacing $(A,B,C,D,E,F)\mapsto(a,b,c,d,e,f)$,
and utilizing the balancing condition 
$q^{1-n}abc=def$.
Note that one can write \eqref{cWqW:1}, \eqref{cWqW:2} as
equivalent expressions using the balancing condition as follows
\begin{eqnarray}
&&\hspace{-0.35cm}\qhyp43{q^{-n},a,b,c}{d,e,f}{q,q}=
\frac{\left(\frac{de}{ab},\frac{de}{ac};q\right)_{n}}
{\left(\frac{de}{a},\frac{de}{abc};q\right)_{n}}
{}_8W_7\left(\frac{de}{qa};q^{-n},\frac{d}{a},\frac{e}{a},b,c;q,\frac{qa}{f}\right)\\
&&\hspace{3.4cm}=q^{\binom{n}{2}}\left(\frac{-de}{bc}\right)^n\!\!
\frac{(\frac{de}{qa},\frac{d}{a},\frac{e}{a},b,c;q)_n}
{(\frac{de}{qa};q)_{2n}(\frac{de}{abc},e,d;q)_n}\nonumber\\
&&
\hspace{4.4cm}\times\Whyp87{\frac{q^{1-2n}a}{de}}{q^{-n},\frac{q^{1-n}}{d},\frac{q^{1-n}}{e},
\frac{q^{1-n}ab}{de},\frac{q^{1-n}ac}{de}}{q,\frac{qa}{f}},
\end{eqnarray}
which reduces the number of inequivalent expressions by two.
This completes the proof.
\end{proof}
}

\medskip

\section{{The symmetric structure of  terminating representations of the Askey--Wilson polynomials}}

{
In this section we describe the
symmetric nature of the equivalence classes of
expressions for the terminating basic hypergeometric representations which 
correspond to the 
Askey--Wilson polynomials.}

{
Consider the 11 equivalence classes of
terminating ${}_4\phi_3$ and ${}_8W_7$
expressions in Corollaries~\ref{cor:3.3}-\ref{WatqWhipp2},
namely \eqref{aw:def4to5}--\eqref{cor3.5a.6b}.
There are four equivalence classes of 
balanced terminating ${}_4\phi_3$ expressions 
\eqref{cor3.5a.3}--\eqref{cor3.5a.6b}
and 7 equivalence classes of very-well-poised terminating
${}_8W_7$ expressions 
\eqref{aw:def4to5}--\eqref{cor3.5a.7c}.
Equivalent expressions 
within an equivalence class are obtained by
compositions of the trivial interchange of positions for numerator and/or denominator parameters in the basic hypergeometric 
series and under the 4!=24 permutations of the symmetric parameter 
$c,d,e,f$ labeling.}

{
The above described 11 equivalence classes 
in Corollaries \ref{cor:3.3}-\ref{WatqWhipp2} correspond to a 
total of 7 equivalence classes of terminating
basic hypergeometric
series representations of the Askey--Wilson
polynomials. These are represented
by 3 ${}_4\phi_3$ equivalence classes
and 4 ${}_8W_7$ equivalence classes
which are given in 
Theorem~\ref{thm:3.1}. Note that each of these
equivalence classes are equivalent under 
the map $\theta\mapsto-\theta$.
}

{
In this section we describe the symmetric
nature of these equivalence classes under
the mapping of inversion \eqref{inversion}
and that due to a theorem due to 
Van der Jeugt and Rao \cite{VanderJeugtRao1999}
which provides the symmetry group of 
nonterminating very well poised ${}_8W_7$
basic hypergeometric functions, namely
Theorem \ref{VanderJeugt43} below.
The symmetry groups of several
relevant basic
hypergeometric functions have been studied 
in the literature
\cite{Kajihara2014,KrattenthalerRao2004,Lievensetal2007,
VanderJeugtRao1999}.
For terminating balanced ${}_4\phi_3$ expressions, the following surprisingly simple result has
been established in
\cite[Proposition 2]{VanderJeugtRao1999}.
}

{
\begin{thm}[Van der Jeugt and Rao (1999)]
Let $n\in\N_0$, 
$q\in\CCdag$,
%$q\in\CC^\ast$, such that $|q|\ne 1$,
${\bf x}:=\{x_1,x_2,x_3,x_4,x_5,x_6\}$, $x_k\in\CCast$, $k\in
%{\bf 6}:=
\{1,2,3,4,5,6\}$, 
be six parameters satisfying
$x_{123456}=q^{1-n}$,
with 
$f:\CCast^6\times\CCdag\to\CC$ defined by
\begin{equation}
f({\bf x};q):=
q^{\binom{n}{2}}
\frac{(x_{1234},x_{1235},x_{1236};q)_n}
{x_{123}^n}
\qhyp43{q^{-n},x_{23},x_{13},x_{12}}
{x_{1234},x_{1235},x_{1236}}{q,q}.
\end{equation}
Then $f({\bf x})$ is symmetric in the variables $x_k$.
\label{VanderJeugt43}
\end{thm}
}

{
From Van der Jeugt and Rao's (1999) result, it is clear that the symmetry
group of the terminating balanced 
${}_4\phi_3$ is $S_6$, the symmetric group
of degree six, $|S_6|=720$. This was originally
established by \cite{Beyeretal1987}, although the 720 transformations were explicitly
written out by Bailey
\cite[Chapter VII]{Bailey64}.
Upon examination of the 
terminating balanced ${}_4\phi_3$ expressions
in Corollary \ref{WatqWhipp2}, we see that there
are four equivalence classes of basic
hypergeometric representations for these
expressions \eqref{cor3.5a.3}--\eqref{cor3.5a.6b}}. 

\begin{rem}
The Van der Jeugt and Rao (1999) result \cite[Proposition 2]{VanderJeugtRao1999} clearly
indicates that the symmetry group structure of the terminating balanced ${}_4\phi_3$ is $S_6$, {which has order equal to} 720. It is interesting to make comparison
of this result with the four terminating balanced expressions in Corollary \ref{WatqWhipp2}, namely 
\eqref{cor3.5a.3}--\eqref{cor3.5a.6b}. 
\end{rem}

%We now analyze
%the number of possibilities that one 
%has within these expressions.  

{
\begin{prop}
\label{firstS6}
The number of allowed permutations 
and rearrangements of the terminating
balanced ${}_4\phi_3$s
\eqref{cor3.5a.3}--\eqref{cor3.5a.6b} in
Corollary \ref{WatqWhipp2} is $|S_6|=720$ (where $|\cdot|$ represents the cardinality).
\label{VDJ4phi3}
\end{prop}
\begin{proof}
There are 6 possible variable pair product combinations 
$(cd, ce, cf, de, df, ef)$.
In what proceeds, we ignore the positioning of the numerator factor $q^{-n}$.
For \eqref{cor3.5a.3}, \eqref{cor3.5a.4} there are 6 possible numerator positionings for each pair combination and 6 possible denominator positionings for each pair combination, so $|\eqref{cor3.5a.3}|=6^3=216$. Therefore
$|{\eqref{cor3.5a.3}, \eqref{cor3.5a.4}}|=432$. For \eqref{cor3.5a.5}, \eqref{cor3.5a.6b}, there are four variables, $(c,d,e,f)$ and again 6 possible numerator positionings and 6 possible denominator positionings, so 
$|\eqref{cor3.5a.5}|=6\times 6\times 4=144$.
Since \eqref{cor3.5a.6b} is the inversion of
\eqref{cor3.5a.5}, the counting is the same.
Hence, $|\eqref{cor3.5a.5}, \eqref{cor3.5a.6b}|=288$. Finally we have  $|\eqref{cor3.5a.3}, \eqref{cor3.5a.4},
\eqref{cor3.5a.5},
\eqref{cor3.5a.6b}|=432+288=720=|S_6|.$
\end{proof}
}

{
\begin{rem}
There is no symmetry analysis
for a terminating ${}_8W_7$ which 
corresponds to the Van der Jeugt and Rao (1999)
result for a terminating balanced ${}_4\phi_3$.
They do however have a symmetry proposition
for a nonterminating ${}_8W_7$, namely 
\cite[Proposition 5]{VanderJeugtRao1999}, see
Theorem \ref{VanderJeugt87} below. It is 
important to note that the nonterminating 
${}_8W_7$ does not possess Gasper and Rahman's inversion
symmetry, Theorem \ref{thm:2.2}, and there is
no nonterminating analog of this symmetry,
so the group structure of the terminating
${}_8W_7$ is not necessarily clear. On the
other hand, one has the Watson $q$-analog of
Whipple's theorem 
\cite[(17.9.16)]{NIST:DLMF} which relates
a terminating balanced ${}_4\phi_3$
to a terminating very-well-poised ${}_8W_7$,
so one expects there to be a one-to-one
relation between these functions. 
\end{rem}
}

\noindent We now
prove this result.

{
\begin{prop}
\label{secS6}
The number of allowed permutations 
and rearrangements of the terminating
balanced ${}_8W_7$s
\eqref{aw:def4to5}--\eqref{cor3.5a.7c} in
Corollary \ref{cor:3.3} is $|S_6|=720$.
\label{VDJ8W7}
\end{prop}
\begin{proof}
As in Proposition \ref{VDJ4phi3}, ignore the positioning of the numerator factor $q^{-n}$.
For \eqref{aw:def4to5}, \eqref{cor3.5a.1}, there are $4!=24$ permutations
of the variables $c,d,e,f$. For \eqref{cor3.5a.6}, there are four 
triple-variable product combinations $(cde, ced, cdf, def)$, 
and therefore the number of possibilities 
for each of the 24 possibilities. Hence  
$|\eqref{cor3.5a.6}|=24\times 4=96$. Its inversion pair 
\eqref{cor3.5a.7c}
has the same number of possibilities, namely 96.
For \eqref{cor3.5a.7} one has 4 variables
with four possible three-variable
product combinations, for each of the four three-variable
product combinations, there are 4 possible numerator
parameter positions for the $cdef$ term, and 6 possible
arrangements of the three remaining variables. Hence
there are 24 possible positionings of the numerator parameters.
Again with four possible three-variable product combinations
$(cde, ced, cfd, def)$, we arrive again at 96, and as well for
its inversion pair \eqref{cor3.5a.7b}, so 
$|\eqref{cor3.5a.7}, \eqref{cor3.5a.7b}, \eqref{cor3.5a.6}, \eqref{cor3.5a.7c}|=96\times 4=384$.
For \eqref{cor3.5a.2}, which is its own self-inverse, 
we have 48 possibilities. Since there are 6 two-variable
product combinations $(cd, ce, cf, de, de, ef)$, then
one has $|\eqref{cor3.5a.2}|=46\times 6=288$.
Summing up the contributions one has
$|\eqref{aw:def4to5}, \eqref{cor3.5a.1}, \eqref{cor3.5a.2}, \eqref{cor3.5a.7},
\eqref{cor3.5a.7b}, \eqref{cor3.5a.6}, \eqref{cor3.5a.7c}|=24\times 2+96\times 4+288=720=|S_6|$.
This completes the proof.
\end{proof}
}

\begin{table}[h]
\centering
\begin{tabular}{|c||c|c|c|c||c|c|c|c|c|c|c|} \hline
\mlt{{\sc Expression}\\{\sc equivalence}\\{\sc class}} 
\TT\TB 
&\eqref{cor3.5a.3}
&\eqref{cor3.5a.4}
&\eqref{cor3.5a.5}
&\eqref{cor3.5a.6b}
&\eqref{aw:def4to5}
&\eqref{cor3.5a.1}
&\eqref{cor3.5a.2}
&\eqref{cor3.5a.7}
&\eqref{cor3.5a.7b}
&\eqref{cor3.5a.6}
&\eqref{cor3.5a.7c}
\\
\hline  
\mlt{{\sc Number of \TT}\\{\sc possibilities\TB}}
\TT\TB 
&216&216&144&144
&{\bf 24}
&{\bf 24}
&{\bf 288}
&{\bf 96}
&{\bf 96}
&{\bf 96}
&{\bf 96}\\
\hline
\end{tabular}
\caption{
{
Total number of arrangements for
terminating balanced ${}_4\phi_3$s \eqref{cor3.5a.3}--\eqref{cor3.5a.6b} and terminating very-well-poised 
${}_8W_7$s \eqref{aw:def4to5}--\eqref{cor3.5a.7c} expressions (in bold) in Corollaries \ref{cor:3.3}-\ref{WatqWhipp2}. The total number
of possibilities, namely the possible arrangements and relabelings,
sum separately to the order $|S_6|=720$,
namely for each set of equivalence classes of ${}_4\phi_3$s and ${}_8W_7$s separately. See Propositions
\ref{VDJ4phi3}, \ref{VDJ8W7}.
}
}
\label{Tableposs}
\end{table}

\noindent See Table \ref{Tableposs} for a delineation
of the total number of
possibilities of expressions
in Corollaries \ref{cor:3.3}-\ref{WatqWhipp2}.

\begin{rem}
Even though the set of transformations for the terminating balanced ${}_4\phi_3$s and ${}_8W_7$s each correspond to the symmetric group $S_6$, the breakdown of equivalence classes does not appear to be isomorphic to any of the subgroups of $S_6$ that the authors investigated. However there are many subgroups of $S_6$ (1455) \cite{GP}, so future investigations may provide some insight here.
\end{rem}

{
\begin{rem}
A straightforward 
analysis of the transformations implied by
Theorem \ref{VanderJeugt43}, indicates
that under these transformations,
each of the four equivalence classes of the balanced ${}_4\phi_3$ expressions in 
Corollary \ref{WatqWhipp2} maps using
Theorem \ref{VanderJeugt43} separately 
to all three other equivalence classes, 
see Figure \ref{figinvinv}.
\end{rem}
}

{
\begin{rem}
Observe that the ${}_4\phi_3$
equivalence classes of expressions
\eqref{cor3.5a.3}--\eqref{cor3.5a.6b} in
Corollary \ref{WatqWhipp2} are paired
\eqref{cor3.5a.3}$\leftrightarrow$\eqref{cor3.5a.4}
and 
\eqref{cor3.5a.5}$\leftrightarrow$\eqref{cor3.5a.6b}
under Gasper and Rahman's inversion 
formula, $z=q$, $r=3$ in \eqref{cor:2.4:r1},
for a terminating basic hypergeometric ${}_4\phi_3$ representation of the Askey--Wilson
polynomial, 
\begin{equation}
\qhyp{4}{3}{q^{-n},a_1,a_2,a_3}{b_1,b_2,b_3}{q,q}\!=\!
q^{-\binom{n}{2}}
(-1)^n
\frac{(a_1,a_2,a_3;q)_n}
{(b_1,b_2,b_3;q)_n}
\qhyp{4}{3}{q^{-n},
\frac{q^{1-n}}{b_1},
\frac{q^{1-n}}{b_2},
\frac{q^{1-n}}{b_3}}
{
\frac{q^{1-n}}{a_1},
\frac{q^{1-n}}{a_2},
\frac{q^{1-n}}{a_3}}
{q,q},
\end{equation}
where $q^{1-n}a_{123}=b_{123}$.
Furthermore, the ${}_8W_7$ equivalence classes of expressions
\eqref{aw:def4to5}--\eqref{cor3.5a.1} are paired using Gasper and
Rahman's inversion formula, namely 
the equality \eqref{aw:def4to5}$=$\eqref{cor3.5a.1}.
See the shaded regions and thick arrows 
in Figure \ref{figinv} for a pictorial representation
of these inversion pairings.
\end{rem}
}

{
\begin{rem}
One can study the mappings of the equivalence classes of 
expressions in Corollaries \ref{cor:3.3}-\ref{WatqWhipp2} to the terminating representations
of the Askey--Wilson polynomials in Theorem \ref{AWthm}
by using the standard map 
\begin{equation}
(b,c,d,e,f)\mapsto\left(q^{-n}\expe^{2i\theta}, a_p\expe^{i\theta}, a_r\expe^{i\theta}, a_t\expe^{i\theta}, a_u\expe^{i\theta}
\right).
\label{stand}
\end{equation}
Both expressions 
\eqref{cor3.5a.3}, \eqref{cor3.5a.4},
map to the basic hypergeometric representation
\eqref{aw:def3}, except with \eqref{cor3.5a.4},
one has $\theta\mapsto-\theta$.
For the ${}_4\phi_3$ expressions under the standard map \eqref{stand}, the expression \eqref{cor3.5a.5} maps to \eqref{aw:def2}
and the expression \eqref{cor3.5a.6b} maps to \eqref{aw:def1}.
Similarly for the ${}_8W_7$ expressions using \eqref{stand}, then \eqref{aw:def4to5}, \eqref{cor3.5a.1} ($\theta\mapsto-\theta$) map to 
\eqref{aw:def4}; \eqref{cor3.5a.2} maps to \eqref{aw:def5}; \eqref{cor3.5a.6}, \eqref{cor3.5a.7b} ($\theta\mapsto-\theta$) maps to \eqref{aw:def6}; and \eqref{cor3.5a.7}, \eqref{cor3.5a.7c} ($\theta\mapsto-\theta$) maps to \eqref{aw:def7}.
See Figure \ref{figinv} for a pictorial representation of these mappings 
from Corollaries \ref{cor:3.3}-\ref{WatqWhipp2} to the terminating
representations of the Askey--Wilson polynomials in Theorem \ref{AWthm}.
\end{rem}
}

\usetikzlibrary{arrows,automata,positioning}
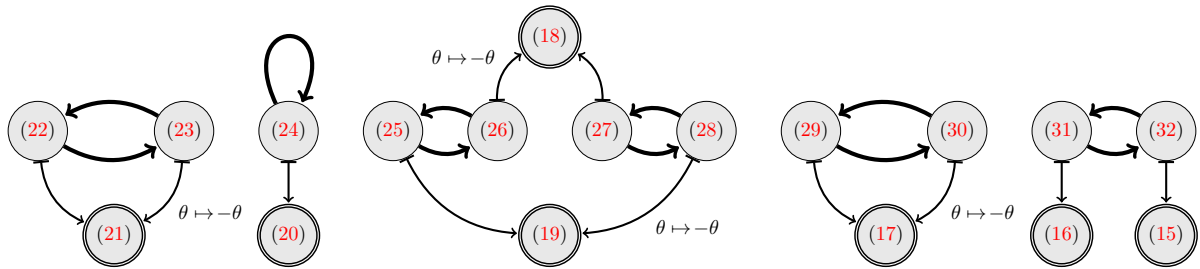
\begin{figure}[H]
\centering
\hspace{-0.0cm}\resizebox{\textwidth}{!}
{
\begin{tikzpicture}[shorten >=1pt,node distance=2cm,on grid,auto]
  \tikzstyle{every state}=[fill={rgb:black,1;white,10}]
  \tikzset{every loop/.style={min distance=10mm,in=65,out=115,looseness=10}}

    \node[state] (r_34) at (2.0,0.0) {\eqref{cor3.5a.2}};

    \node[state] (r_33) [left of=r_34] {\eqref{cor3.5a.1}};
    \node[state] (r_32) [left = 2.8cm of r_33] {\eqref{aw:def4to5}};

    %(28) aw:def6
    %(29) aw:def7
    %(35) cor3.5a.7   -> (37) cor3.5a.6
    %(36) cor3.5a.7b  -> (38) cor3.5a.7c
    %(37) cor3.5a.6   -> (35) cor3.5a.7
    %(38) cor3.5a.7c  -> (36) cor3.5a.7b 

    \node[state]   (r_35)  [right of=r_34]    {\eqref{cor3.5a.7}};
    \node[state]   (r_38)  [right of=r_35]    {\eqref{cor3.5a.7b}};
    \node[state]   (r_36)  [right of=r_38]    {\eqref{cor3.5a.6}};
    \node[state]   (r_37)  [right of=r_36]    {\eqref{cor3.5a.7c}};

    \node[state]   (r_39)  [right = 2.0cm of r_37]    {\eqref{cor3.5a.3}};
    \node[state]   (r_40)  [right = 2.8cm of r_39]    {\eqref{cor3.5a.4}};
    \node[state]   (r_41)  [right of=r_40]    {\eqref{cor3.5a.5}};
    \node[state]   (r_42)  [right of=r_41]    {\eqref{cor3.5a.6b}};

    \node[state,accepting,thick]   (q_27)  [below of=r_42]    {\eqref{aw:def1}};
    \node[state,accepting,thick]   (q_28)  [below of=r_41]    {\eqref{aw:def2}};
    \node[state,thick,accepting] (q_29)  at (13.435,-2.0) {\eqref{aw:def3}};
    \node[state,accepting,thick]   (q_32)  [below of=r_34]    {\eqref{aw:def5}};

    %\node[state,thick,accepting](q_33)[below right = 2.0cm of r_32]{\eqref{aw:def4}};
    \node[state,thick,accepting] (q_33)  at (-1.34,-2.0) {\eqref{aw:def4}};
    
    \node[state,accepting,thick] (q_30) at (7.0,-2.0) {\eqref{aw:def7}};
    \node[state,accepting,thick] (q_31) at (7.0,1.8) {\eqref{aw:def6}};

    \path[->, line width=0.08cm]    (r_32) edge [bend right] node {}    (r_33);
    \path[->, line width=0.08cm]    (r_33) edge [bend right] node {}    (r_32);
    \path[->, line width=0.08cm]    (r_34) edge [loop above] node {}    ();
    \path[->, line width=0.08cm]    (r_35) edge [bend right] node {}    (r_38);
    \path[->, line width=0.08cm]    (r_38) edge [bend right] node {}    (r_35);
    \path[->, line width=0.08cm]    (r_36) edge [bend right] node {}    (r_37);
    \path[->, line width=0.08cm]    (r_37) edge [bend right] node {}    (r_36);
    \path[->, line width=0.08cm]    (r_39) edge [bend right] node {}    (r_40);
    \path[->, line width=0.08cm]    (r_40) edge [bend right] node {}    (r_39);
    \path[->, line width=0.08cm]    (r_41) edge [bend right] node {}    (r_42);
    \path[->, line width=0.08cm]    (r_42) edge [bend right] node {}    (r_41);
    \path[|->,line width=0.04cm]    (r_42) edge              node {}    (q_27);
    \path[|->,line width=0.04cm]    (r_41) edge              node {}    (q_28);
    \path[|->,line width=0.04cm]    (r_39) edge [bend right] node {}    (q_29);
    \path[|->,line width=0.04cm]    (r_40) edge [bend left] node {$\theta\mapsto\!-\theta$}    (q_29);
    \path[|->,line width=0.04cm]    (r_32) edge [bend right] node {}    (q_33); 
    \path[|->,line width=0.04cm]    (r_33) edge [bend left] node {$\theta\mapsto\!-\theta$}    (q_33);
    \path[|->,line width=0.04cm]    (r_34) edge              node {}    (q_32);
    \path[|->,line width=0.04cm]    (r_35) edge [bend right] node {}    (q_30);
    \path[|->,line width=0.04cm]    (r_37) edge [bend left] node {$\theta\mapsto\!-\theta$}    (q_30);
    \path[|->,line width=0.04cm]    (r_36) edge [bend right] node {}    (q_31);
    \path[|->,line width=0.04cm]    (r_38) edge [bend left] node {$\theta\mapsto\!-\theta$}    (q_31);

    %\draw [line width=60pt,double distance=-50pt,opacity=0.1,gray,line cap=round] (r_36) -- (r_37);
    %\draw [line width=60pt,double distance=-50pt,opacity=0.1,gray,line cap=round] (r_35) -- (r_38);
    %\draw [line width=60pt,double distance=-50pt,opacity=0.1,gray,line cap=round] (r_32) -- (r_33);
    %\draw [line width=60pt,double distance=-50pt,opacity=0.1,gray,line cap=round] (r_39) -- (r_40);
    %\draw [line width=60pt,double distance=-50pt,opacity=0.1,gray,line cap=round] (r_41) -- (r_42);
    %\filldraw[color=gray!5,fill=gray,opacity=0.1] (r_34) circle (7.0ex);   
\end{tikzpicture}
}

\caption{{
This figure depicts {the} equivalence classes of {terminating} ${}_8W_7$ \eqref{aw:def4to5}--\eqref{cor3.5a.7c} and ${}_4\phi_3$
\eqref{cor3.5a.3}--\eqref{cor3.5a.6b} expressions 
in Corollaries \ref{cor:3.3}-\ref{WatqWhipp2} {and their
corresponding equivalence classes of terminating Askey--Wilson basic hypergeometric
representations in Theorem \ref{thm:3.1}, \eqref{aw:def1}--\eqref{aw:def4}}.
{The} expressions {\eqref{aw:def4to5}--\eqref{cor3.5a.6b}}
are paired ({using thick arrows}) using Gasper and Rahman's inversion 
formula \eqref{inversion}. More specifically, to verify the 
inversion pairings for the ${}_4\phi_3$
expressions, one can use \eqref{cor:2.4:r1}, and for the 
${}_8W_7$ expressions, one can use \eqref{inv87}, or more 
explicitly the equality of \eqref{aw:def4to5} and \eqref{cor3.5a.1}.
Note that \eqref{cor3.5a.2} is the sole expression which is its own self-inverse.
{For the terminating Askey--Wilson
hypergeometric representation equivalence classes}
\eqref{aw:def1}--\eqref{aw:def4},
{arrows}
indicate which expressions in Corollaries \ref{cor:3.3}-\ref{WatqWhipp2} {are mapped to} 
under the standard map \eqref{stand} to the terminating
representations of the Askey--Wilson polynomials in Theorem \ref{AWthm}.
Arrows marked $\theta\mapsto-\theta$ indicate that
the expressions {in Corollaries \ref{cor:3.3}-\ref{WatqWhipp2}}
map to the same {terminating} Askey--Wilson basic hypergeometric 
representation {equivalence class} under this mapping.}}
\label{figinv}
\end{figure}

\medskip 

{
Now consider the equivalence classes of 
terminating ${}_8W_7$ expressions
in Corollary \ref{cor:3.3}, namely
\eqref{aw:def4to5}--\eqref{cor3.5a.7c}.
There is a surprising structure to the
behavior under mappings of these equivalence
classes. 
%See Figures \ref{figinv}, \ref{figinvinv}.
%\noindent ====\moro{\bf [HSC: placeholder, do not delete. Still under construction.}]\\
Let us start this discussion by reviewing what is known about the symmetry
of the nonterminating ${}_8W_7$.
For nonterminating very-well-poised ${}_8W_7$ expressions, 
the following result has been previously established in
\cite[Proposition 5]{VanderJeugtRao1999}.
}

{
\begin{thm}[Van der Jeugt and Rao (1999)]
Let $q\in\CCdag$,
${\bf x}:=\{x_1,x_2,x_3,x_4,x_5\}$, $x_k\in\CCast$, $k\in\{0,1,2,3,4,5\}$, 
be six parameters
with 
$f:\CCast^6\times\CCdag\to\CC$ defined by
\begin{equation}
f(x_0;{\bf x};q):=
w\!\left(q^{-1}x_0^3x_{12345};
\frac{x_{012345}}{x_1^2},
\frac{x_{012345}}{x_2^2},
\frac{x_{012345}}{x_3^2},
\frac{x_{012345}}{x_4^2},
\frac{x_{012345}}{x_5^2};q\right),
\end{equation}
where
\begin{equation}
w(b;a,c,d,e,f;q)=
\frac{(
\frac{q^2b^2}{acdef},
\frac{qb}{a},
\frac{qb}{c},
\frac{qb}{d},
\frac{qb}{e},
\frac{qb}{f}
;q)_\infty}{(qb;q)_\infty}
\Whyp87{b}{a,c,d,e,f}{q,\frac{q^2b^2}{acdef}}.
\end{equation}
Then $f(x_0;{\bf x};q)$ satisfies
$f(x_0;{\bf x};q)=f(x_0;p\cdot{\bf x};q)$,
for every element $p\in WB_5$ that has an even number of minus signs in its matrix 
representation. 
Hence the invariance group of the 
very-well-poised nonterminating ${}_8W_7$ is the group $WD_5$.
\label{VanderJeugt87}
\end{thm}
}

{
Note that the the groups
$WB_n$ and $WD_n$ are the Weyl groups of the root systems of types
$B_n$ and $D_n$ (see \cite[Chapter III]{Humphreys78}). It is clear from Van der Jeugt and Rao's (1999) discussion that the symmetry
group of the nonterminating very-well-poised ${}_8W_7$ is $WD_5$,
%the Weyl group of the root system of type $D_5$, 
$|WD_5|=5!2^4=1920$.
According to Zudilin \cite{Zudilin2019}
this transformation group was clear in Bailey (1964) \cite[Section 7.5]{Bailey64} which
focused on a study of 
the transformations of the 
very-well-poised nonterminating ${}_7F_6$, whose $q$-analog is
the nonterminating very-well-poised ${}_8W_7$.
(See Zudilin \cite[Lemma 8]{Zudilin2004} for a discussion of the 
computation of the order and some properties of this symmetry group which 
is connected to the group structure of the Riemann zeta value $\zeta(3)$.)
}

%\noindent \moro{\bf [Michael Schlosser:
%There is also a relevant paper by S. Lievens
%and J. Van der Jeugt, ``Symmetry groups of Bailey's
%transformations for $_{10}\phi_9$-series".
%J. Comput. Appl. Math. 206 (2007), no. 1, 498–519,
%which you should look at.]}

\medskip
{
Now we discuss the symmetric nature of the
{terminating} ${}_8W_7$s in Corollary \ref{cor:3.3}.  Terminating ${}_8W_7$ expressions may be obtained from nonterminating ${}_8W_7$ expressions by setting one of the numerator parameters equal to $q^{-n}$, $n\in\N_0$. If you apply 
Van der Jeugt and Rao's Theorem \ref{VanderJeugt87} with one of the
numerator parameters equal to some $q^{-n}$,
then some subset of the transformations map
to equivalence classes for terminating expressions, and the complement maps to equivalence classes of nonterminating 
expressions (not explicitly treated in this paper). The result of the mappings using Theorem \ref{VanderJeugt87} from terminating ${}_8W_7$ equivalence classes
to terminating ${}_8W_7$ equivalence classes is listed 
in Table \ref{8W7to8W7} and displayed pictorially in
Figure \ref{figinvinv}.
}

%\begin{table}[h]
%\centering
% \begin{tabular}{||c|c||} \hline
%\mlt{{\sc Original ${}_8W_7$}\\ {\sc Expression}\\ {\sc Equivalence}\\{\sc Class}}  \TT\TB &
%\mlt{{\sc Mapped ${}_8W_7$}\\ {\sc Expression}\\ {\sc Equivalence}\\{\sc Class{es}}}\\
%\hline 
%\hline \{\eqref{aw:def4to5}\} & \{\moro{\bf 120}\eqref{aw:def4to5}, \eqref{cor3.5a.6}\}\\ 
%\hline\{\eqref{cor3.5a.6}\} & \{\eqref{aw:def4to5}, \eqref{cor3.5a.6}\}\\\hline
%\hline\{\eqref{cor3.5a.1}\} & \{\eqref{cor3.5a.1}, \eqref{cor3.5a.7b}\}\\
%\hline\{\eqref{cor3.5a.7b}\} & \{\eqref{cor3.5a.1}, \eqref{cor3.5a.7b}\}\\\hline
%\hline\{\eqref{cor3.5a.2}\} & \{\eqref{cor3.5a.2}, \eqref{cor3.5a.7}, \eqref{cor3.5a.7c}\}\\
%\hline\{\eqref{cor3.5a.7}\} & \{\eqref{cor3.5a.2}, \eqref{cor3.5a.7}, \eqref{cor3.5a.7c}\}\\
%\hline\{\eqref{cor3.5a.7c}\} & \{\eqref{cor3.5a.2}, \eqref{cor3.5a.7}, \eqref{cor3.5a.7c}\}\\  \hline
%\end{tabular}
%\caption{
%{
%This table lists all the
%mappings of equivalence classes which occur if one applies Theorem \ref{VanderJeugt87}
%with original and resulting ${}_8W_7$ expressions being terminating.
%}
%}
%\label{8W7to8W7}
%\end{table}

\begin{table}[!hbt]
\centering
%\begin{tabular}{||c|c|c|c|c||} \hline
%\mlt{{\sc Original ${}_8W_7$}\\ {\sc Expression}\\ {\sc Equivalence}\\{\sc Class}}  \TT\TB &
%\mlt{{\sc Mapped ${}_8W_7$}\\ {\sc Expression}\\ {\sc Equivalence}\\{\sc Class{es}}}&
%\eqref{aw:def4to5}&\eqref{cor3.5a.6}\\
%\hline 
%\hline \{\eqref{aw:def4to5}\} & \{\eqref{aw:def4to5}, \eqref{cor3.5a.6}\}
%&\moro{\bf 120}& \moro{\bf 480}\\ 
%\hline\{\eqref{cor3.5a.6}\} & \{\eqref{aw:def4to5}, \eqref{cor3.5a.6}\}
%&\moro{\bf 120}& \moro{\bf 480}\\ \hline
%&&\eqref{cor3.5a.1}&\eqref{cor3.5a.7b}\\
%\hline\{\eqref{cor3.5a.1}\} & \{\eqref{cor3.5a.1}, \eqref{cor3.5a.7b}\}
%&\moro{\bf 120}& \moro{\bf 480}\\
%\hline\{\eqref{cor3.5a.7b}\} & \{\eqref{cor3.5a.1}, \eqref{cor3.5a.7b}\}
%&\moro{\bf 120}& \moro{\bf 480}\\ \hline
%\end{tabular}
\begin{tabular}{||c|c||c|c||c|c||c|c|c||} \hline
\mlt{{\sc Original ${}_8W_7$}\\ {\sc Expression}\\ {\sc Equivalence}\\{\sc Class}}  \TT\TB &
\mlt{{\sc Mapped ${}_8W_7$}\\ {\sc Expression}\\ {\sc Equivalence}\\{\sc Class{es}}} & \eqref{aw:def4to5}& \eqref{cor3.5a.6} & \eqref{cor3.5a.1} &\eqref{cor3.5a.7b}&
\eqref{cor3.5a.2}&\eqref{cor3.5a.7}&\eqref{cor3.5a.7c}\\ \hline
\hline \{\eqref{aw:def4to5}\} & \{\eqref{aw:def4to5}, \eqref{cor3.5a.6}\}
&{\bf 120}& {\bf 480} & -- & -- & -- & -- & -- \\ 
\hline\{\eqref{cor3.5a.6}\} & \{\eqref{aw:def4to5}, \eqref{cor3.5a.6}\}
&{\bf 120}& {\bf 480} & -- & -- & -- & -- & -- \\ 
\hline\{\eqref{cor3.5a.1}\} & \{\eqref{cor3.5a.1}, \eqref{cor3.5a.7b}\}
& -- & -- & {\bf 120}& {\bf 480} & -- & -- & -- \\
\hline\{\eqref{cor3.5a.7b}\} & \{\eqref{cor3.5a.1}, \eqref{cor3.5a.7b}\}
& -- & -- & {\bf 120}& {\bf 480} & -- & -- & -- \\ 
\hline
\{\eqref{cor3.5a.2}\} & \{\eqref{cor3.5a.2}, \eqref{cor3.5a.7}, \eqref{cor3.5a.7c}\}& -- & -- & -- & -- & {\bf 120}&{\bf 360}&{\bf 120} \\
\hline
\{\eqref{cor3.5a.7}\} & \{\eqref{cor3.5a.2}, \eqref{cor3.5a.7}, \eqref{cor3.5a.7c}\}& -- & -- & -- & -- & {\bf 120}&{\bf 360}&{\bf 120}\\
\hline
\{\eqref{cor3.5a.7c}\} & \{\eqref{cor3.5a.2}, \eqref{cor3.5a.7}, \eqref{cor3.5a.7c}\}& -- & -- & -- & -- &
{\bf 120}&{\bf 360}&{\bf 120} \\  \hline
\end{tabular}
\caption{
{
Given that the 
original and mapped ${}_8W_7$ expressions are terminating, 
this table provides the mapping properties of the 
${}_8W_7$ equivalence classes
\eqref{aw:def4to5}--\eqref{cor3.5a.7c}  under the action of Theorem \ref{VanderJeugt87}.}
{The numbers on the right-part of the
table indicate the total number of ${}_8W_7$ expressions which are 
mapped using Theorem \ref{VanderJeugt87},
given a specific choice of parameter labeling (dashes represent zero). 
See Figure \ref{figinvinv} for 
a graphical representation of these
mapping properties.}
}
\label{8W7to8W7}
\end{table}

\begin{figure}[H]
\centering
\hspace{-0.0cm}\resizebox{0.75\textwidth}{!}
{
\begin{tikzpicture}
[shorten >=1pt,node distance=5cm,on grid,auto]
%[node distance=0pt, nodes={outer sep=0pt, minimum width=3cm, thick, minimum height=1cm}]
  \tikzstyle{every state}=[fill={rgb:black,1;white,10}]
  \tikzset{every loop/.style={min distance=10mm,in=155,out=205,looseness=15}}

    \node[state] (r_34) at (-2.65,1.7) {\eqref{cor3.5a.2}};
%    \node[state] (r_36) at (0,2.4) [above above right = 2.2cm of r_34]  {\eqref{cor3.5a.7}};
    %\node[state] (r_36) at (0.0,2.4) {\eqref{cor3.5a.7}};
        \node[state] (r_36) at (0.0,3.4) {\eqref{cor3.5a.7}};

%    \node[state] (r_38)  [below below right = 2.2cm of r_34] {\eqref{cor3.5a.7c}};
    \node[state] (r_38) at (0.0,0.0) {\eqref{cor3.5a.7c}};
    \node[state] (r_37)  [right = 2.7cm of r_36] {\eqref{cor3.5a.7b}};
    \node[state] (r_35)  [right = 2.7cm of r_38] {\eqref{cor3.5a.6}};
    \node[state] (r_33)  [right = 2.7cm of r_37] {\eqref{cor3.5a.1}};
    \node[state] (r_32)  [right = 2.7cm of r_35] {\eqref{aw:def4to5}};

    \path[->, line width=0.07cm]  (r_34) edge [loop left] node {} ();
    
    %\path[<->, line width=0.07cm] (r_37) edge [] node {} (r_36);
    \path[->,line width=0.07cm] (r_37) edge [bend right=15] node {} (r_36);
    \path[->,line width=0.07cm] (r_36) edge [bend right=15] node {} (r_37);
    %\path[<->, line width=0.07cm] (r_38) edge [] node {} (r_35);
    \path[->,line width=0.07cm] (r_38) edge [bend right=15] node {} (r_35);
    \path[->,line width=0.07cm] (r_35) edge [bend right=15] node {} (r_38);
    %\path[<->, line width=0.07cm] (r_33) edge [] node {} (r_32);
    \path[->,line width=0.07cm] (r_33) edge [bend right=15] node {} (r_32);
    \path[->,line width=0.07cm] (r_32) edge [bend right=15] node {} (r_33);
  
    %\path[<->,line width=0.035cm] (r_32) edge [] node {} (r_35);
    \path[->,line width=0.035cm] (r_32) edge [bend right=13] node {} (r_35);
    \path[->,line width=0.035cm] (r_35) edge [bend right=13] node {} (r_32);
   %\path[<->,line width=0.035cm] (r_33) edge [] node {} (r_37);
    \path[->,line width=0.035cm] (r_33) edge [bend right=13] node {} (r_37);
    \path[->,line width=0.035cm] (r_37) edge [bend right=13] node {} (r_33);
    %\path[<->,line width=0.035cm] (r_36) edge [] node {} (r_38);
    \path[->,line width=0.035cm] (r_36) edge [bend right=13] node {} (r_38);
    \path[->,line width=0.035cm] (r_38) edge [bend right=13] node {} (r_36);
    %\path[<->,line width=0.035cm] (r_36) edge [] node {} (r_34);
    \path[->,line width=0.035cm] (r_34) edge [bend right=13] node {} (r_36);
    \path[->,line width=0.035cm] (r_36) edge [bend right=13] node {} (r_34);
    %\path[<->,line width=0.035cm] (r_34) edge [] node {} (r_38);
    \path[->,line width=0.035cm] (r_34) edge [bend right=13] node {} (r_38);
    \path[->,line width=0.035cm] (r_38) edge [bend right=13] node {} (r_34);
    
    %\draw [line width=35pt,double distance=-5pt,opacity=0.1,gray,line cap=round] (r_36) -- (r_37);
     %\draw [line width=35pt,double distance=-5pt,opacity=0.1,gray,line cap=round] (r_35) -- (r_38);
     %  \draw [line width=35pt,double distance=-5pt,opacity=0.1,gray,line cap=round] (r_32) -- (r_33);
    %\filldraw[color=gray!5,fill=gray,opacity=0.25] (r_34) circle (7.0ex); 
    
 \draw[rounded corners,fill=none,line width=35pt,opacity=0.1,double distance=5pt,line cap=round] (r_34) -- (r_36) -- (r_38) -- (r_34) -- (r_36);
  \draw[rounded corners,fill=none,line width=35pt,opacity=0.1,double distance=5pt,line cap=round] (r_37) -- (r_33) -- (r_37);
      \draw[rounded corners,fill=none,line width=35pt,opacity=0.1,double distance=5pt,line cap=round] (r_35) -- (r_32) -- (r_35);
\end{tikzpicture}
}

\caption{{This figure provides a graphical representation 
of Table \ref{8W7to8W7} together with
the action of inversion \eqref{inversion}.
More specifically, it depicts the equivalence classes of
terminating very-well-poised ${}_8W_7$
expressions
\eqref{aw:def4to5}--\eqref{cor3.5a.7c}
in Corollary \ref{cor:3.3}, with
thick arrows indicating pairings using
inversion \eqref{inversion},
and  thin arrows indicating mappings using Theorem \ref{VanderJeugt87}. The shaded 
regions indicate equivalence class 
grouping under Theorem \ref{VanderJeugt87}.
}}
\label{figinvinv}
\end{figure}
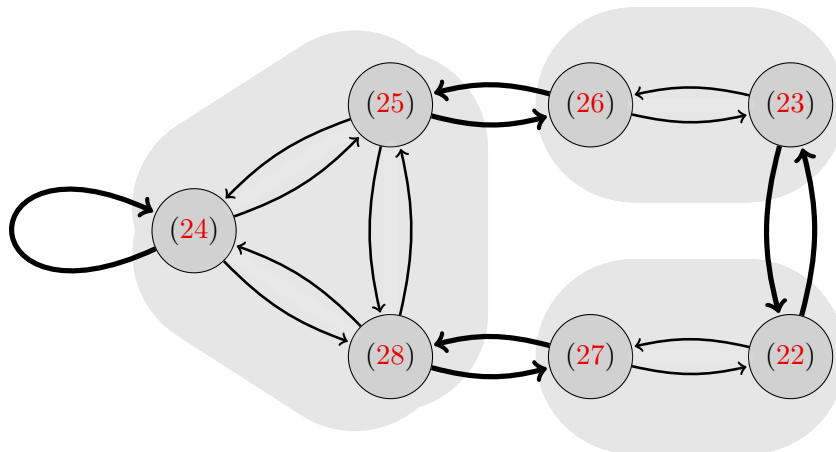

\medskip
{
Upon examination of the 
terminating balanced ${}_8W_7$ expressions
in Corollary \ref{cor:3.3}, we see that there
are seven equivalence classes of terminating ${}_8W_7$
expressions \eqref{aw:def4to5}--\eqref{cor3.5a.7c}. A straightforward 
computer algebra analysis of the transformations implied by 
Van der Jeugt and Rao's nonterminating Proposition, Theorem \ref{VanderJeugt87}
(where we have selected only those expressions which result in terminating
expressions), 
has indicated that under these transformations,
each of the seven equivalence classes of the very-well-poised ${}_8W_7$ expressions split into three separate associations of terminating 
very-well-poised ${}_8W_7$ equivalence classes.
}

{
\begin{rem}
The three separate associations of equivalence classes for terminating
very-well-poised ${}_8W_7$s in Corollary \ref{cor:3.3} which are
obtained by applying Theorem \ref{VanderJeugt87} for nonterminating
${}_8W_7$ misses the connections between the three associations. In order
to connect these associations, one must rely on Gasper and Rahman's inversion
formula, which have no nonterminating counterpart,
so therefore would be undiscoverable using Van der Jeugt and Rao's (1999)
analysis \cite[Proposition 5]{VanderJeugtRao1999}.
\end{rem}
}

%{
%\begin{rem}
%\begin{eqnarray}
%&&\{\eqref{aw:def4to5}\}\mapsto\{\eqref{cor3.5a.3}\},\nonumber\\[-0.05cm]
%&&\{\eqref{cor3.5a.1}\}\mapsto\{\eqref{cor3.5a.4}\},\nonumber\\[-0.05cm]
%&&\{\eqref{cor3.5a.2}\}\mapsto\{\eqref{cor3.5a.3}, \eqref{cor3.5a.5}\},\nonumber\\[-0.05cm]
%&&\{\eqref{cor3.5a.6}\}\mapsto\{\eqref{cor3.5a.3}, \eqref{cor3.5a.5}\},\nonumber\\[-0.05cm]
%&&\{\eqref{cor3.5a.7}\}\mapsto\{\eqref{cor3.5a.3}, \eqref{cor3.5a.6b}\},\nonumber\\[-0.05cm]
%&&\{\eqref{cor3.5a.7b}\}\mapsto\{\eqref{cor3.5a.4}, \eqref{cor3.5a.5}\},\nonumber\\[-0.05cm]
%&&\{\eqref{cor3.5a.7c}\}\mapsto\{\eqref{cor3.5a.4}, \eqref{cor3.5a.6b}\}.\nonumber
%\end{eqnarray}
%\end{rem}
%}

\begin{table}[!hbt]
\centering
\begin{tabular}{||c|c;{1pt/1pt}c|c|c;{1pt/1pt}c|c;{1pt/1pt}c||}
\hline
\mlt{{\sc Original ${}_4\phi_3$}\\ {\sc Expression}\\ {\sc Equivalence}} \TT\TB 
&  &  &  &  & & & \\
{\sc Class} & \eqref{aw:def4to5} & \eqref{cor3.5a.1} & \eqref{cor3.5a.2} & \eqref{cor3.5a.7} & \eqref{cor3.5a.7b} & \eqref{cor3.5a.6} & \eqref{cor3.5a.7c}  \\
\hline\{\eqref{cor3.5a.3}\}
& {\bf 4} & {\bf 4}  & {\bf 56} & {\bf 20} & {\bf 20} & {\bf 20} & {\bf 20}\\
\hline\{\eqref{cor3.5a.4}\} 
& {\bf 4} & {\bf 4}  & {\bf 56} & {\bf 20} & {\bf 20} & {\bf 20} & {\bf 20}\\
\hline\{\eqref{cor3.5a.5}\}  
& {\bf 6} & {\bf 6}  & {\bf 60} & {\bf 18} & {\bf 18} & {\bf 18} & {\bf 18}\\
\hline\{\eqref{cor3.5a.6b}\}
& {\bf 6} & {\bf 6}  & {\bf 60} & {\bf 18} & {\bf 18} & {\bf 18} & {\bf 18}\\
\hline
\end{tabular}
\caption{
{
This table lists the mappings and their total number of
occurrences which occur if one applies the converse for Watson's
$q$-analog of Whipple's theorem, namely Corollary \eqref{WatqWhipp} to
the terminating balanced ${}_4\phi_3$ expressions in Corollary \ref{WatqWhipp2}.
For each ${}_4\phi_3$ expression,terminating very-well poised ${}_8W_7$ expressions are produced when you include all permutations of the 
numerator parameters and denominator parameters. The numbers on the right-part of the table
indicate the total number of expression equivalence classes (out of a $3!^2=36$ permutations) mapped to for a given choice of parameter labeling.
Dotted lines represent boundaries of inversion
pairs.}
}
\end{table}

\begin{table}[!hbt]%[H]
\centering
 \begin{tabular}{||c|c;{1pt/1pt}c|c;{1pt/1pt}c||} \hline
\mlt{{\sc Original ${}_8W_7$}\\ {\sc Expression}\\ {\sc Equivalence}}  \TT\TB  &&&&\\
{\sc Class} &  \eqref{cor3.5a.3} & \eqref{cor3.5a.4} & \eqref{cor3.5a.5} & \eqref{cor3.5a.6b} \\
\hline
\hline\{\eqref{aw:def4to5}\} & {\bf 24} & {\bf 24} & {\bf 24} & {\bf 24} \\
\hline\{\eqref{cor3.5a.1}\} & {\bf 24}  & {\bf 24} & {\bf 24} & {\bf 24} \\
\hline\{\eqref{cor3.5a.2}\} & 
%\{\eqref{cor3.5a.3}, \eqref{cor3.5a.5}\} 
{\bf 28} & {\bf 28} & {\bf 20} & {\bf 20} \\
\hline\{\eqref{cor3.5a.7}\}  & {\bf 30} & {\bf 30} & {\bf 16} & {\bf 16}\\
\hline\{\eqref{cor3.5a.7b}\} & {\bf 30} & {\bf 30} & {\bf 16} & {\bf 16}\\
\hline\{\eqref{cor3.5a.6}\}  & {\bf 30} & {\bf 30} & {\bf 16} & {\bf 16}\\
\hline\{\eqref{cor3.5a.7c}\} & {\bf 30} & {\bf 30} & {\bf 16} & {\bf 16}\\  \hline
\end{tabular}
\caption{
{
This table lists the
mappings which occur if one applies Watson's $q$-analog of
Whipple's theorem, namely Corollary \ref{WatqWhipp2},
including all permutations of the numerator parameters,
to the terminating very-well poised ${}_8W_7$ expressions
in Corollary \ref{cor:3.3}. This results in the production
of terminating balanced ${}_4\phi_3$s for each
${}_8W_7$ expression. The numbers on the right-part of the table indicate the total number of expression 
equivalence classes (out of $2\cdot4!=48$ permutations) mapped
to give a given choice of parameter labeling.
Dotted lines represent boundaries of inversion
pairs.
}
}
\end{table}

\newpage
\section*{Acknowledgments}
We would like to thank Mourad Ismail, Tom Koornwinder,
Eric Rains, %Michael Schlosser,
Hjalmar Rosengren and Joris Van der Jeugt for valuable discussions.
R.S.C-S acknowledges financial  support  through  the  research  project
PGC2018-096504-B-C33  supported  by  Agencia  Estatal  de Investigaci\'on of Spain.

\begin{table}[H]
\centering
\begin{tabular}{||c|c|c|c|c|c|} \hline
\mlt{{\sc Original ${}_8W_7$}\\ {\sc Expression}\\ {\sc Equivalence}\\{\sc Class}}  \TT\TB &
\mlt{{\sc Mapped ${}_8W_7$}\\ {\sc Expression}\\ {\sc Equivalence}\\{\sc Class{es}}} & \eqref{cWqW:1}& \eqref{cWqW:4} & \eqref{cWqW:2} &\eqref{cWqW:3} \\ \hline
\hline \{\eqref{cWqW:1}\} & \{\eqref{cWqW:1}, \eqref{cWqW:4}\}&360&240& --& --\\ 
\hline\{\eqref{cWqW:4}\} & \{\eqref{cWqW:1}, \eqref{cWqW:4}\}&360&240& --& --\\ 
\hline\{\eqref{cWqW:2}\} & \{\eqref{cWqW:2}, \eqref{cWqW:3}\}& --& --&360&240\\
\hline\{\eqref{cWqW:3}\} & \{\eqref{cWqW:2}, \eqref{cWqW:3}\}& --& --&360&240\\
\hline
\end{tabular}
\caption{
{This table describes mapping properties of the converse for Watson's $q$-analog of Whipple's theorem, Corollary \ref{WatqWhipp}.
It first provides
the mapping properties
for the ${}_8W_7$ equivalence classes 
\eqref{cWqW:1}--\eqref{cWqW:4}  which are mapped if one applies Van der Jeugt and Rao's nonterminating Proposition, Theorem \ref{VanderJeugt87} (where we have selected only those expressions which result in terminating expressions).
The numbers on the right-part of the table
indicate the total number of ${}_8W_7$ expression equivalence classes mapped to for a given choice of parameter labeling. Dashes indicate
zero mappings. 
%\goro{Dotted lines represent boundaries of inversion pairs.}
See 
Figure \ref{figinvinvinv}.}
%\moro{\bf [Michael Schlosser: 
%The first sentence (or more) in the description
%of Table 1 could be improved in language.]}
}
\label{TabcqWat}
\end{table}

%\newpage
\usetikzlibrary{arrows,automata,positioning}
%\begin{figure}[H]
\begin{figure}
\centering
\hspace{-0.0cm}\resizebox{\textwidth}{!}
{
\begin{tikzpicture}[shorten >=1pt,node distance=2cm,on grid, auto]
  \tikzset{every loop/.style={min distance=10mm,looseness=12}}

  \tikzstyle{every state}=[fill={rgb:black,1;white,10}]

  \node[state] (q_78) {\eqref{cWqW:3}};
  % % \node[state,initial] (q_78)  at (0.0,0.0) {\eqref{cWqW:1}};

  \node[state] (q_76) [right of =q_78] {\eqref{cWqW:2}};
  \node[state] (q_77) [right = 3.0cm of q_76] {\eqref{cWqW:1}};
  \node[state] (q_79) [right of =q_77] {\eqref{cWqW:4}};

%[right = 2.8cm of r_39] 

 \path[->, line width=0.07cm]  (q_78) edge [in=150,out=210,loop] node {} ();
  \path[->, line width=0.07cm]  (q_79) edge [in=30,out=330,loop] node {} ();

  \path[->,line width=0.07cm] (q_76) edge [bend right=15] node {} (q_77);
  \path[->,line width=0.07cm] (q_77) edge [bend right=15] node {} (q_76);

  \path[->,line width=0.035cm] (q_78) edge [bend right=15] node {} (q_76);
  \path[->,line width=0.035cm] (q_76) edge [bend right=15] node {} (q_78);

  \path[->,line width=0.035cm] (q_77) edge [bend right=15] node {} (q_79);
  \path[->,line width=0.035cm] (q_79) edge [bend right=15] node {} (q_77);

 \draw[rounded corners,fill=none,line width=35pt,opacity=0.1,double distance=5pt,line cap=round] (q_78) -- (q_76) -- (q_78);
 \draw[rounded corners,fill=none,line width=35pt,opacity=0.1,double distance=5pt,line cap=round] (q_77) -- (q_79) -- (q_77);
\end{tikzpicture}
}
\caption{{This figure depicts the relation of equivalence classes of
terminating very-well-poised ${}_8W_7$
expressions
\eqref{cWqW:1}--\eqref{cWqW:4}
in the converse for Watson's $q$-analog of Whipple's theorem, Corollary \ref{WatqWhipp}. Thick arrows indicate equivalence classes
which are paired
using Gasper and Rahman's inversion
formula for terminating basic hypergeometric series \eqref{inversion}.
Thin arrows indicate which nodes map terminating
${}_8W_7$ equivalence classes to terminating ${}_8W_7$ equivalence classes
using Van der Jeugt and Rao's nonterminating {Proposition}, Theorem \ref{VanderJeugt87}
(where we have selected only those expressions which result in terminating
expressions), 
see Table \ref{8W7to8W7}. Shaded regions indicate
which equivalence classes are grouped using Theorem \ref{VanderJeugt87}.}}
\label{figinvinvinv}
\end{figure}
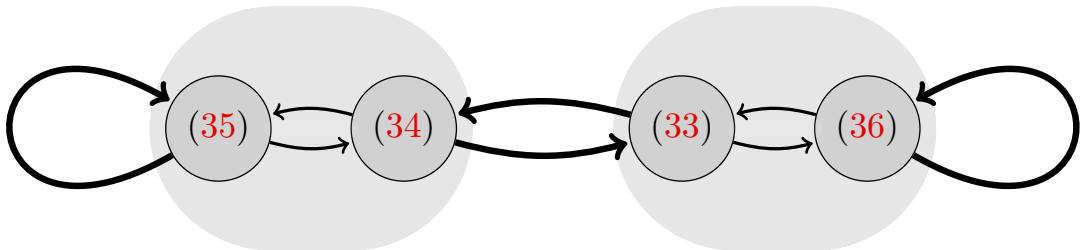

\appendix

\section{Full collections of terminating 4-parameter symmetric interchange transformations}
\label{fullcol}
In this appendix, as a matter of completeness,
we present the entirety of all of the 
parameter interchange transformations for terminating basic 
hypergeometric transformations 
which arise from the Askey--Wilson polynomials.
One may use the transformations presented in 
this subsection to rewrite all the expressions given
in Corollaries \ref{cor:3.3}-\ref{WatqWhipp2}. To learn more about 
the symmetric interchange transformations
for Askey--Wilson polynomials and to see the
proofs of the results presented in this section, see
\cite[Section 3.3]{CohlCostasSantosGe}.

%Note that in Corollary \ref{cor:3.3}, 
The evidence that the first (and second) ${}_8W_7$ in Corollary \ref{cor:3.3} are symmetric
in the variables $c,d,e,f$ is clear. Therefore, all of the
formulas in this corollary are invariant under
the interchange of any two of those variables. This is true whether 
the symmetry between those variables is
evident in the corresponding mathematical expression or not. 
Perhaps, the most famous parameter interchange transformation 
of this sort is Sears' balanced 
${}_4\phi_3$ transformations \cite[(17.9.14)]{NIST:DLMF} which 
demonstrate the invariance (and provide specific transformation
formulas) of the Askey--Wilson polynomials under parameter interchange. 
Other interesting parameter interchange transformations of this type 
can be obtained, such as by \eqref{cor3.5a.2} with
$c\leftrightarrow d$ (preserves the argument), 
$c\leftrightarrow e$, 
$c\leftrightarrow f$, 
$d\leftrightarrow e$, 
$d\leftrightarrow f$ 
interchanged (the invariance under the interchange
$e\leftrightarrow f$ is evident). Furthermore, when 
the symmetry within a set of variables is
evident in the transformation corollaries presented below,
%(e.g., Corollary \ref{cor:4.3} which is symmetric in the variables $c,d,e$;
%Corollary \ref{cor:4.13} which is symmetric
%in the variables $c,d$, etc.), 
then due to this symmetry, non-trivial 
transformation formulas can be obtained
by equating the two expressions with certain variables interchanged. 

\begin{cor}%[Corollary 4 in \cite{CohlCostasSantosGe}]
\cite[Corollary 4]{CohlCostasSantosGe}
\label{cor:3.5}
Let $n\in\N_0$, $b,c,d,e,f\in\CCast$,
$q\in\CCdag$.
Then, one has 
the following
parameter interchange transformations for a terminating ${}_8W_7$:
\vspace{12pt}
\begin{eqnarray}
&&\hspace{-2.2cm}\label{cor3.5:r1}
{}_8W_7\left(\frac{q^{-n}c}{d};q^{-n},\frac{q^{-n}c}{b},
\frac{qb}{de},\frac{qb}{df},c;q,\frac{ef}{b}\right)\\
&&\hspace{-0.2cm}\label{cor3.5:r2}=
\frac{\left(\frac{qb}{de},\frac{qb}{df},\frac{qb}{c},\frac{d}{c},c;q\right)_n}
{\left(\frac{qb}{ce},\frac{qb}{cf},\frac{qb}{d},\frac{c}{d},d;q\right)_n}\,
{}_8W_7\left(\frac{q^{-n}d}{c};q^{-n},\frac{q^{-n}d}{b},
\frac{qb}{ce},\frac{qb}{cf},\boro{d};q,\frac{ef}{b}\right)\\
&&\hspace{-0.2cm}\label{cor3.5:r3}=
\boro{\frac
{\left(\frac{qb}{cd},\frac{qb}{e},\frac{d}{c},e;q\right)_n}
{\left(\frac{qb}{ce},\frac{qb}{d},\frac{e}{c},d;q\right)_n}\,}
{}_8W_7\left(\frac{q^{-n}c}{e};q^{-n},\frac{q^{-n}c}{b},
\frac{qb}{ed},\frac{qb}{ef},c;q,\frac{df}{b}\right)\\
&&\hspace{-0.2cm}\label{cor3.5:r4}=
\boro{\frac{\left(\frac{qb}{ed},\frac{qb}{ef},\frac{qb}{c},
\frac{\boro{d}}{c},c;q\right)_n}
{\left(\frac{qb}{c\boro{e}},\frac{qb}{cf},\frac{qb}{\boro{d}},\frac{c}{e},
\boro{d};q\right)_n}\,}
{}_8W_7\left(\frac{q^{-n}e}{c};q^{-n},\frac{q^{-n}e}{b},
\frac{qb}{cd},\frac{qb}{cf},e;q,\frac{df}{b}\right)\\
&&\hspace{-0.2cm}\label{cor3.5:r5}=
\boro{\frac{\left(\frac{qb}{\boro{cd}},\frac{qb}{\boro{f}},\frac{\boro{d}}{\boro{c}},f;q\right)_n}
{\left(\frac{qb}{\boro{c}f},\frac{qb}{\boro{d}}\frac{f}{c},\boro{d};q\right)_n}\,}
{}_8W_7\left(\frac{q^{-n}c}{f};q^{-n},\frac{q^{-n}c}{b},
\frac{qb}{ef},\frac{qb}{df},c;q,\frac{de}{b}\right)\\
&&\boro{\hspace{-0.2cm}\label{cor3.5:r6}=
\frac{\left(\frac{qb}{fd},\frac{qb}{fe},\frac{qb}{c},\frac{d}{c},c;q\right)_n}
{\left(\frac{qb}{ce},\frac{qb}{cf},\frac{qb}{d},\frac{c}{f},d;q\right)_n}
{}_8W_7\left(\frac{q^{-n}f}{c};q^{-n},\frac{q^{-n}f}{b},
\frac{qb}{cd},\frac{qb}{ce},f;q,\frac{de}{b}\right)}\\
&&\hspace{-0.2cm}\label{cor3.5:r7}\boro{=\frac
{\left(\frac{qb}{ef},\frac{d}{c};q\right)_n}
{\left(\frac{qb}{cf},\frac{d}{e};q\right)_n}
{}_8W_7\left(\frac{q^{-n}e}{d};q^{-n},\frac{q^{-n}e}{b},
\frac{qb}{dc},\frac{qb}{df},e;q,\frac{cf}{b}\right)}\\
&&\hspace{-0.2cm}\label{cor3.5:r8}=\frac
{\left(\frac{qb}{dc},\frac{qb}{df},\boro{\frac{qb}{e}},\frac{d}{c},e;q\right)_n}
{\left(\frac{qb}{ce},\frac{qb}{cf},\frac{qb}{d},\frac{e}{d},d;q\right)_n}
{}_8W_7\left(\frac{q^{-n}d}{e};q^{-n},\frac{q^{-n}d}{b},
\frac{qb}{ec},\frac{qb}{ef},d;q,\frac{cf}{b}\right)\\
&&\hspace{-0.2cm}\label{cor3.5:r9}=
\boro{\frac{\left(\frac{qb}{ef},\frac{d}{c};q\right)_n}
{\left(\frac{qb}{ce},\frac{d}{f};q\right)_n}
{}_8W_7\left(\frac{q^{-n}f}{d};q^{-n},\frac{q^{-n}f}{b},
\frac{qb}{dc},
\frac{qb}{de},
f;q,\frac{ce}{b}\right)}\\
&&\hspace{-0.2cm}\label{cor3.5:r10}=
\frac{\left(\frac{qb}{de},\boro{\frac{qb}{dc},\frac{qb}{f},\frac{d}{c},f};q\right)_n}
{\left(\frac{qb}{ce},\boro{\frac{qb}{cf},\frac{qb}{d}},\frac{f}{d},\boro{d};q\right)_n}
{}_8W_7\left(\frac{q^{-n}d}{f};q^{-n},\frac{q^{-n}d}{b},
\frac{qb}{fc},
\frac{qb}{fe},
d;q,\frac{ce}{b}\right)\\
&&\hspace{-0.2cm}\label{cor3.5:r11}=
\boro{\frac{\left(\frac{qb}{ed},\frac{qb}{f},\frac{d}{c},f;q\right)_n}
{\left(\frac{qb}{cf},\frac{qb}{d},\frac{f}{e},d;q\right)_n}
{}_8W_7\left(\frac{q^{-n}e}{f};q^{-n},\frac{q^{-n}e}{b},
\frac{qb}{fc},\frac{qb}{fd},
e;q,\frac{cd}{b}\right)}\\
&&\hspace{-0.2cm}\label{cor3.5:r12}=
\boro{
\frac{\left(\frac{qb}{df},\frac{qb}{e},\frac{d}{c},e;q\right)_n}
{\left(\frac{qb}{ce},\frac{qb}{d},\frac{e}{f},d;q\right)_n}
{}_8W_7\left(\frac{q^{-n}f}{e};q^{-n},\frac{q^{-n}f}{b},
\frac{qb}{ec},\frac{qb}{ed},
f;q,\frac{cd}{b}\right)}
.
\end{eqnarray}
\end{cor}
%\begin{proof}
%\medskip
%\noindent
%Start with \eqref{cor3.5a.2} and consider %all permutations of the symmetric parameters %$c,d,e,f$ which produce non-trivial %transformations. The ordering of the %elements is given
%by the first argument of the ${}_8W_7$ as %follows:~$
%\{(c,d),(d,c),(c,e),(e,c),(c,f),(f,c), ..., (e,f),(f,e)\}.$
%\end{proof}

\begin{cor}\cite[Corollary 5]{CohlCostasSantosGe}
\label{cor:3.6}
Let $n\in\N_0$, $b,c,d,e,f\in\CCast$, 
$q\in\CCdag$.
Then, one has 
the following parameter interchange transformations for a terminating ${}_8W_7$:
\vspace{12pt}
\begin{eqnarray}
&&\hspace{-2.2cm}\label{cor3.6:r1}
{}_8W_7\left(\frac{qb^2}{def};q^{-n},\frac{qb}{de},
\frac{qb}{df},\frac{qb}{ef},c;q,\frac{q^{n+1}{b}}{c}\right)
\\
&&\hspace{-0.2cm}\label{cor3.6:r2}
=\frac{\left(\frac{qb}{c},\frac{q^2b^2}{def};q\right)_n}
{\left(\frac{qb}{d},\frac{q^2b^2}{cef};q\right)_n}
{}_8W_7\left(\frac{qb^2}{cef};q^{-n},\frac{qb}{ce},
\frac{qb}{cf},\frac{qb}{ef},d;q,\frac{q^{n+1}{b}}{d}\right)
\\
&&\hspace{-0.2cm}\label{cor3.6:r3}
=\frac{\left(\frac{qb}{c},\frac{q^2b^2}{def};q\right)_n}
{\left(\frac{qb}{e},\frac{q^2b^2}{cdf};q\right)_n}
{}_8W_7\left(\frac{qb^2}{cdf};q^{-n},\frac{qb}{cd},
\frac{qb}{cf},\frac{qb}{df},e;q,\frac{q^{n+1}{b}}{e}\right)
\\
&&\hspace{-0.2cm}\label{cor3.6:r4}
=\frac{\left(\frac{qb}{c},\frac{q^2b^2}{def};q\right)_n}
{\left(\frac{qb}{f},\frac{q^2b^2}{cde};q\right)_n}
{}_8W_7\left(\frac{qb^2}{cde};q^{-n},
\frac{qb}{cd},\frac{qb}{ce},\frac{qb}{de},f;q,\frac{q^{n+1}{b}}{f}\right).
\end{eqnarray}
\end{cor}

\begin{cor}\cite[Corollary 6]{CohlCostasSantosGe}\label{cor3.8}
Let $n\in\N_0$, $b,c,d,e,f\in\CCast$, 
$q\in\CCdag$.
Then, one has 
the following parameter interchange transformations
for a terminating ${}_4\phi_3$:
\begin{eqnarray}
&&\hspace{-2.2cm}\label{cor3.8:r1}\qhyp43{q^{-n},\frac{qb}{ef},c,d}
{\frac{q^{-n}cd}{b},\frac{qb}{e},\frac{qb}{f}}{q,q}\\
&&\hspace{-0.2cm}\label{cor3.8:r2}=
\frac{\left(\frac{qb}{de},\frac{qb}{c};q\right)_n}
{\left(\frac{qb}{cd},\frac{qb}{e};q\right)_n}\qhyp43{q^{-n},
\frac{qb}{cf},d,e}{\frac{q^{-n}de}{b},\frac{qb}{c},\frac{qb}{f}}{q,q}\\
&&\hspace{-0.2cm}\label{cor3.8:r3}=
\frac{\left(\frac{qb}{df},\frac{qb}{c};q\right)_n}
{\left(\frac{qb}{cd},\frac{qb}{f};q\right)_n}\qhyp43{q^{-n},\frac{qb}{ce},d,f}
{\frac{q^{-n}df}{b},\frac{qb}{c},\frac{qb}{e}}{q,q}\\
&&\hspace{-0.2cm}\label{cor3.8:r4}=
\frac{\left(\frac{qb}{ce},\frac{qb}{d};q\right)_n}
{\left(\frac{qb}{cd},\frac{qb}{e};q\right)_n}\qhyp43{q^{-n},
\frac{qb}{df},c,e}{\frac{q^{-n}ce}{b},\frac{qb}{d},\frac{qb}{f}}{q,q}\\
&&\hspace{-0.2cm}\label{cor3.8:r5}=
\frac{\left(\frac{qb}{cf},\frac{qb}{d};q\right)_n}
{\left(\frac{qb}{cd},\frac{qb}{f};q\right)_n}\qhyp43{q^{-n},\frac{qb}{de},c,f}
{\frac{q^{-n}cf}{b},\frac{qb}{d},\frac{qb}{e}}{q,q}\\
&&\boro{\hspace{-0.2cm}\label{cor3.8:r6}=
\frac{\left(\frac{qb}{ef},\boro{\frac{qb}{c}},\frac{qb}{d};q\right)_n}
{\left(\frac{qb}{de},\boro{\frac{qb}{e}},\frac{qb}{f};q\right)_n}\qhyp43{q^{-n},
\frac{qb}{cd},e,f}{\frac{q^{-n}ef}{b},\frac{qb}{c},\frac{qb}{d}}{q,q}.}
\end{eqnarray}
\end{cor}

\begin{proof}
\noindent \boro{
Start with \eqref{cor3.5a.3} and consider all permutations of the symmetric parameters $c,d,e,f$ which produce non-trivial transformations.
}
\end{proof}

%\begin{rem}
%Another set of parameter interchange transformations
%can be obtained by considering all permutations
%of the symmetric parameters $c,d,e,f$ in
%\eqref{cor3.5a.4}. However, one can see that
%these are equivalent to the above Corollary 
%\ref{cor3.8} by replacing
%\[
%(b,c,d,e,f)\mapsto\left(\frac{q^{-2n}}{b},\frac{q^{-n}c}{b},\frac{q^{-n}d}{b},\frac{q^{-n}e}{b},\frac{q^{-n}f}{b}\right).
%\]
%\end{rem}

\boro{
\begin{cor}\cite[Corollary 7]{CohlCostasSantosGe}
\label{cor3.10}
Let $n\in\N_0$, $b,c,d,e,f\in\CCast$, 
$q\in\CCdag$.
Then, one has 
the following parameter interchange transformations for a terminating ${}_4\phi_3$:
\begin{eqnarray}
&&\hspace{-2.2cm}\label{cor3.10:r1}
\qhyp43{q^{-n},\frac{qb}{cd},\frac{qb}{ce},\frac{qb}{cf}}
{\frac{q^2b^2}{cdef},\frac{q^{1-n}}{c},\frac{qb}{c}}{q,q}\\
&&\hspace{-0.2cm}\label{cor3.10:r2}=
\frac{\left(\frac{qb}{d},d;q\right)_n}
{\left(\frac{qb}{c},c;q\right)_n}
\qhyp43{q^{-n},\frac{qb}{dc},\frac{qb}{de},\frac{qb}{df}}
{\frac{q^2b^2}{cdef},\frac{q^{1-n}}{d},\frac{qb}{d}}{q,q}\\
&&\hspace{-0.2cm}\label{cor3.10:r3}=
\frac{\left(\frac{qb}{e},e;q\right)_n}
{\left(\frac{qb}{c},c;q\right)_n}
\qhyp43{q^{-n},\frac{qb}{ec},\frac{qb}{ed},\frac{qb}{ef}}
{\frac{q^2b^2}{cdef},\frac{q^{1-n}}{e},\frac{qb}{e}}{q,q}\\
&&\hspace{-0.2cm}\label{cor3.10:r4}=
\frac{\left(\frac{qb}{f},f;q\right)_n}
{\left(\frac{qb}{c},c;q\right)_n}
\qhyp43{q^{-n},\frac{qb}{fc},\frac{qb}{fd},\frac{qb}{fe}}
{\frac{q^2b^2}{cdef},\frac{q^{1-n}}{f},\frac{qb}{f}}{q,q}.
\end{eqnarray}
\end{cor}
}
\begin{proof}
%noindent \boro{
Start with \eqref{cor3.5a.5} and consider all permutations of the symmetric parameters $c,d,e,f$ which produce non-trivial transformations.
%}
\end{proof}

%\begin{rem}
%Another set of parameter interchange transformations
%can be obtained by considering all permutations
%of the symmetric parameters $c,d,e,f$ in
%\eqref{cor3.5a.6b}. However, one can see that
%these are equivalent to the above Corollary 
%\ref{cor3.10} by replacing
%\[
%(b,c,d,e,f)\mapsto\left(\frac{q^{-n}f}{e},\frac{qb}{ce},\frac{qb}{d%e},f,\frac{q^{-n}f}{b}\right).
%\]
%\end{rem}

%\noindent \moro{\bf
%[Michael Schlosser:~Are you sure that none of the listed
%transformations repeat? The point is that (55) is different
%from (43) (and (65) different from (59)), so it might be
%that parts of Corollary 14 (resp.~Corollary 16)
%are already given in Corollary 13 (resp. Corollary 15)
%by a change of variables.]
%}

%\section{Some nonterminating basic hypergeometric transformations}
%\setcounter{equation}{0}
%%%%%%%%%%%%%%%%%%%%%%%%%%%%%%%%%%%%%%%%%%%%%%%%%%%%%%%%%%%%%%%%%%%%%%%%%%%%%%%%%%%
%\renewcommand\theequation{A.\arabic{equation}}
%\renewcommand\thethm{A.\arabic{thm}}
%\renewcommand\therem{A.\arabic{rem}}
%\renewcommand\thethm{A.\arabic{thm}}
%%%%%%%%%%%%%%%%%%%%%%%%%%%%%%%%%%%%%%%%%%%%%%%%%%%%%%%%%%%%%%%%%%%%%%%%%%%%%%%%%%%%

%\newpage
%%%%\section*{References}
%\bibliographystyle{plain}
%\bibliography{qtransAW2}
%\bibliography{refbib}
%%\bibliography{../refbib}   % Roberto bib
%\bibliography{/home/hcohl/tex/refbib}   % Howard bib yes!

\def\cprime{$'$} \def\dbar{\leavevmode\hbox to 0pt{\hskip.2ex \accent"16\hss}d}

\end{document}